\documentclass[12pt,reqno]{amsart}

\numberwithin{equation}{section}

\usepackage{amssymb}
\usepackage{amsmath}
\usepackage{amsbsy}
\usepackage{amscd}
\usepackage{amsthm}
\usepackage{amsfonts}
\usepackage{enumerate}

\usepackage{ifpdf}
\ifpdf \usepackage[pdftex,pdfstartview=FitH,pdfpagemode=none,colorlinks,bookmarks,linkcolor=blue]{hyperref} \else  \usepackage[hypertex]{hyperref} \fi

\usepackage[nobysame,abbrev,alphabetic]{amsrefs}

\usepackage[nohug,heads=vee]{diagrams}
\diagramstyle[textstyle=\scriptstyle,labelstyle=\scriptstyle]

\newtheorem{theorem}{Theorem}[section]

\newtheorem{lemma}[theorem]{Lemma}
\newtheorem{corollary}[theorem]{Corollary}
\newtheorem{definition}[theorem]{Definition}

\newtheorem{proposition}[theorem]{Proposition}
\newtheorem{notation}[theorem]{Notation}

\newtheorem{remark}[theorem]{Remark}
\theoremstyle{definition}

\newcommand{\cA}{\mathcal{A}}

\newcommand{\cE}{\mathcal{E}}
\newcommand{\cF}{\mathcal{F}}

\newcommand{\cL}{\mathcal{L}}
\newcommand{\cM}{\mathcal{M}}

\newcommand{\cO}{\mathcal{O}}

\newcommand{\cS}{\mathcal{S}}

\newcommand{\bC}{\mathbb{C}}

\newcommand{\bR}{\mathbb{R}}
\newcommand{\bZ}{\mathbb{Z}}
\newcommand{\bQ}{\mathbb{Q}}

\newcommand{\bN}{\mathbb{N}}

\newcommand{\bT}{\mathbb{T}}

\newcommand{\goz}{\mathfrak{z}}
\newcommand{\gow}{\mathfrak{w}}

\newcommand{\Kb}{{\overline K}}
\newcommand{\Fb}{{\overline F}}

\newcommand{\GL}{\operatorname{GL}}

\newcommand{\re}{\operatorname{Re}}
\newcommand{\im}{\operatorname{Im}}
\newcommand{\rank}{\operatorname{rank}}

\newcommand{\Aut}{\operatorname{Aut}}
\newcommand{\Tr}{\operatorname{Tr}}

\newcommand{\Spec}{\operatorname{Spec}}
\newcommand{\Gal}{\operatorname{Gal}}

\newcommand{\Vol}{\mathrm{Vol}}
\newcommand{\Mah}{\mathrm{Mah}}
\newcommand\diag[1]{\operatorname{diag}\left(#1\right)}
\newcommand{\onto}{\xymatrix{\ar@{>>}[r]&}}
\newcommand{\da}[4]{\xymatrix{#1 \ar@<.5ex>[r]^{#2} \ar@<-.5ex>[r]_{#3} & #4}}
\newcounter{subconst}[subsection]

\newcounter{const}

\newcounter{CONST}

\begin{document}

\title[Remarks on Euclidean minima]{Remarks on Euclidean minima}
\author[U. Shapira]{Uri Shapira}
\address{ETH Z\"{u}rich, 8092 Z\"{u}rich, Switzerland}
\author[Z. Wang]{Zhiren Wang}
\address{Yale University, New Haven, CT 06520, USA}
\setcounter{page}{1}
\thanks{Uri Shapira is partially supported by the Advanced research Grant 228304 from the European Research Council. Zhiren Wang is partially supported by an AMS-Simons travel grant and the NSF Grant DMS-1201453.}

\begin{abstract}
The Euclidean minimum $M(K)$ of a number field $K$ is an important numerical invariant that indicates whether $K$ is norm-Euclidean. When $K$ is a non-CM field of unit rank $2$ or higher, Cerri showed $M(K)$, as the supremum in the Euclidean spectrum $\Spec(K)$, is isolated and attained and can be computed in finite time. We extend Cerri's works by applying recent dynamical results of Lindenstrauss and Wang. In particular, the following facts are proved:\\
(1) For any number field $K$ of unit rank $3$ or higher, $M(K)$ is isolated and attained and Cerri's algorithm computes $M(K)$ in finite time.\\
(2) If $K$ is a non-CM field of unit rank $2$ or higher, then the computational complexity of $M(K)$ is bounded in terms of the degree, discriminant and regulator of $K$.

\end{abstract}
\maketitle
{\small\tableofcontents}

\section{Introduction}

\subsection{Background}

A number field $K$ is said to be norm-Euclidean if its ring of integers $\cO_K$ is a Euclidean domain with respect to the algebraic norm $|N_K(\cdot)|$, that is, for all $x,y\in\cO_K$, there exists $a\in\cO_K$ such that $|N_K(x-ay)|<|N_K(y)|$. The Euclidean minimum of $K$ is a numerical indicator of whether $K$ is norm-Euclidean or not.

\begin{definition} The {\bf Euclidean minimum} of an element $x\in K$ is
$m_K(x)=\inf_{y\in x+\cO_K}|N_K(y)|$.

The {\bf Euclidean spectrum} of the number field $K$ is the image $\{m_K(x):x\in K\}$ and the {\bf Euclidean minimum} of $K$ is
$M(K)=\sup_{x\in K}m_K(x)$.
\end{definition}

It is known that $K$ is norm-Euclidean if $M(K)>1$ and is not norm-Euclidean if $M(K)<1$. When $M(K)=1$, it was proved by Cerri \cite{C06} that if the unit rank of $K$ is at least $2$ then it is not norm-Euclidean.

One can easily check that $m_K(x)\geq 0$ and $M(K)>0$. When $K$ is totally real it is part of a conjecture of Minkowski that $M(K)\leq 2^{-d}\sqrt{D_K}$ where $d$ and $D_K$ denote respectively the degree and discriminant of $K$. The conjecture has been proved only for number fields of low degrees. But in general, weaker upper bounds are available: for totally real fields Chebotarev proved $M(K)\leq 2^{-\frac d2}\sqrt{D_K}$ (see for example \cite{HW79}*{\S 24.9}). For general number fields (not necessarily totally real), Bayer-Fluckiger showed in \cite{B06} that
\begin{equation}\label{EMupper}M(K)\leq 2^{-d}D_K.\end{equation}
In the rest of this paper, we will always write
\begin{equation}\label{normKb}\Kb=K\otimes_\bQ\bR.\end{equation} The notions of Euclidean spectrum and minimum can be extended to $\Kb$. To do this, one needs a natural extension of the algebraic norm $N_K$ to a continuous function $N_\Kb$ on the real vector space $\Kb$, which satisfies
\begin{equation}\label{inhomonorm}N_\Kb(x\otimes s)=s^dN_K(x),\forall x\in K,s\in\bR.\end{equation}
The exact definition of $N_\Kb$ will be given in Definition \ref{normKbdef}.

\begin{definition}
For $x\in\Kb$, we define the {\bf inhomogeneous minimum} of $x$ as
$m_\Kb(x)=\inf_{y\in x+\cO_K}|N_\Kb(y)|$.
The {\bf inhomogeneous spectrum} and {\bf inhomogeneous minimum} of $K$ are respectively\\ $\Spec(\Kb)=\{m_\Kb(x):x\in\Kb\}$ and $M(\Kb)=\sup_{x\in\Kb}m_\Kb(x)$.
\end{definition}

$K$ is a subset of $K\otimes\bR$ by identifying $x$ with $x\otimes1$. Moreover, $N_\Kb$ and $m_\Kb$ coincide respectively with $N_K$ and $m_K$ when restricted to $K$. In consequence, $\Spec(K)\subset\Spec(\Kb)$ and $M(K)\leq M(\Kb)$. Actually the equality between the two minima always holds by the works of Barnes-Swinnerton-Dyer \cite{BS52}, van der Linden \cite{V85} and Cerri \cite{C06}.

\begin{proposition}\label{IEmin}\cite{C06}*{Corollary 1}

$M(\Kb)=M(K)$ for all number fields $K$. If the unit rank of $K$ is at least $2$, then $M(K)\in\bQ$.\end{proposition}

It can be shown that for all $x\in K$, \begin{equation}\label{attain}\exists y\in x+\cO_K\text{ such that } |N_\Kb(y)|\leq M(K).\end{equation}

\begin{definition}$M(K)$ is said to be {\bf attained}
in $\Spec(\Kb)$ if (\ref{attain}) is true for all $x\in\Kb$ as well; or, equivalently, whenever $x\in\Kb$ satisfies $m_\Kb(x)=M(K)$, there exists $y\in x+\cO_K$ such that $|N_\Kb(y)|$ is exactly $M(K)$.\end{definition}

It should be noted that the term ``attained" is defined in a slightly confusing sense. It doesn't refer to whether or not the supremum is achieved in the sup-inf expression \begin{equation}M(K)=M(\Kb)=\sup_{x\in\Kb}\inf_{y\in x+\cO_K}|N_\Kb(y)|.\end{equation} Instead, it means the infimum is attained at every point $x$ that achieves the supremum.

For more information on Euclidean minima and spectra, we refer the reader to Lemmermeyer's survey \cite{L95}.

A few natural questions one can ask about $M(K)$ are:
\begin{enumerate}
\item\label{q1} Is $M(K)\in \Spec(K)$?
\item\label{q2} If $M(K)\in \Spec(K)$, is it an isolated point in $\Spec(K)$?
\item\label{q3} Is $M(K)$ attained in the sense defined above?
\item\label{q4} Is $\Spec(K)$ equal to $\Spec(\Kb)$? Can one provide concrete description of these spectra?
\item\label{q5} Is $M(K)$ algorithmically computable?
\item\label{q6}  Can one bound the computational complexity of $M(K)$?
\end{enumerate}
We will see that questions~\eqref{q1}--\eqref{q4} are very much related and a complete answer to question~\eqref{q4} usually allows one to answer the preceding ones.

In relation to questions~\eqref{q1}--\eqref{q4}, when $K$ is a non-CM field of unit rank at least $2$, Cerri  proved in \cite{C06} that $M(K)$ is isolated and attained in $\Spec(\Kb)$. In fact, $\Spec(K)$ and $\Spec(\Kb)$ were shown to be equal and a complete characterization of the spectra was obtained. Namely, the non-zero part of $\Spec(K)$ is a decreasing sequence of rational numbers that converge to $0$. Furthermore, Cerri showed that the preimage $\{x\in\Kb:m_\Kb(x)=M(K)\}$ is a non-empty subset of $K$ and is the union of finitely many residue classes modulo $\cO_K$.

In relation to question~\eqref{q5},
computations in many fields of low degree are listed in \cite{L95}.
By developing ideas introduced in works of Barnes and Swinnerton-Dyer \cite{BS52} and Cavallar and Lemmmermeyer \cite{CL98}, which focused respectively on quadratic and cubic fields, Cerri \cites{CThesis,C07} gave an algorithm that computes $M(K)$. In case that $K$ is not a CM number field and has unit rank $2$ or higher, Cerri showed that the algorithm always terminates in a finite number of steps. However, it is unknown whether the algorithm works in general. Moreover, there was no bound on the number of steps required before the algorithm stops, i.e. the computational complexity of $M(K)$.

An important ingredient in the works \cites{CThesis,C06,C07} of Cerri was the application of a result from dynamical systems by Berend to the natural multiplicative action by the group of units $U_K$ on $\Kb/\cO_K$. In \cite{B83}, Berend established the rigidity of higher-rank irreducible commutative actions by toral automorphisms, showing that any orbit is finite or dense. The special form needed for the application in this case will be stated in Theorem \ref{Berend}. The higher-rank and irreducibility conditions in Berend's theorem correspond respectively to Cerri's assumptions that the unit rank is at least $2$ and the number field is non-CM.

\subsection{Statement of main results}

In this paper, we strengthen and complement Cerri's results in two different ways based on two recent extensions to Berend's theorem.

In contrast to the irreducible actions studied in Theorem \ref{Berend}, Lindenstrauss and Wang investigated in \cite{LW10} a special case of reducible actions by commuting toral automorphisms. Namely, given a totally irreducible Cartan action by automorphisms on the $d$-dimensional torus $\bT^d$, the meaning of which will be specified later in the paper, one can consider the diagonal action on $(\bT^d)^2$. When the rank of the action is at least $3$, orbit closures under the diagonal action are classified in \cite{LW10}. When $K$ is a CM field, the action by $U_K$ on $\Kb/\cO_K$ is, up to a finite lifting, such a diagonal action (see Diagram \ref{commdiag}).

Using the classification above, when $K$ is CM and $\rank(U_K)\geq 3$, the main result of the present paper, Theorem \ref{describespec}, allows us to complement Cerri's work  and
answer questions~\eqref{q1}--\eqref{q5}. It gives a complete description of the inhomogeneous and Euclidean spectra, which will actually be proved to be equal and contain only rational numbers. In contrast to the non-CM case, the spectra have an infinity of accumulation points. In particular, the following corollaries follow from Theorem \ref{describespec}. The following corollary positively answers questions~\eqref{q1}--\eqref{q4}.

\begin{corollary}\label{isolatta} Suppose $K$ is a CM number field of unit rank $3$ or higher. Then $M(K)$ is attained and isolated in $\Spec(\Kb)$; moreover, the set $\{z\in\Kb: m_\Kb(z)=M(K)\}$ is contained in $K$ and is the union of finitely many residue classes modulo $\cO_K$.
\end{corollary}
In relation to question~\eqref{q5} we prove

\begin{corollary}\label{computable} If $K$ is a CM number field of unit rank $3$ or higher, then Cerri's algorithm computes $M(K)$ in a finite number of steps.
\end{corollary}

Combined with Cerri's results \cites{CThesis,C06,C07} regarding non-CM fields, it follows that the properties in Corollaries \ref{isolatta} and \ref{computable} hold, or equivalently, questions \eqref{q1}--\eqref{q5} have affirmative answers, for any number field that has unit rank $3$ or higher, and in particular, for all fields of degree $7$ or higher.

\begin{theorem}Assume $K$ is a number field of unit rank $3$ or higher. Then $M(K)$ is attained and isolated in $\Spec(\Kb)$, and is computable in finite time.\end{theorem}

Finally, in relation to question~\eqref{q6}, in \cite{W10}, an effective version of Theorem \ref{Berend} was obtained by generalizing methods from  Bourgain, Lindenstrauss, Michel and Venkatesh's one-dimensional study \cite{BLMV09}. Using this result and following Cerri's strategies, we will give in Section~\ref{seccomplexity} an estimate of the number of possible residue classes in $K$ modulo $\cO_K$ on which $m_K$ may possibly achieve the maximum $M(K)$ when $K$ is non-CM and $U_K$ has at least rank $2$. This yields the following upper bound to the computational complexity of $M(K)$.

\begin{theorem}\label{complexity}
If $K$ is a non-CM number field whose unit rank is greater than or equal to 2, then the computational complexity of $M(K)$ is bounded by $\exp\exp\exp(D_K^{C\cF_{U_K}^2})$ where the constant $C$ depends only on the degree of $K$.
\end{theorem}

Here $D_K$ is the discriminant of the field and $\cF_{U_K}$ is a number that measures the sizes of a set of fundamental units in $K$, for the exact definition, see \S\ref{subsecbounded}. In particular, $\cF_{U_K}=O_d(R_K)$ where $d$ and $R_K$ are respectively the degree and regulator of $K$. So Theorem \ref{complexity} says the computational complexity of the Euclidean minimum is bounded in terms of the degree, discriminant and regulator of the number field.

Due to the fact that the orbit closure classification from \cite{LW10} is ineffective, such a complexity estimate is currently unavailable for CM fields. Any quantitative version of that result (which is
stated in the present paper as Proposition \ref{joining}) would lead to some kind of bound on the computational complexity of $M(K)$, $K$ being CM with unit rank $3$ or higher.

\section{Generalities}\label{general}

\subsection{Notations in number fields}

From now on let $K$ be a number field with $r_1$ real embeddings $\sigma_1,\cdots,\sigma_{r_1}$ and $r_2$ pairs of conjugate imaginary embeddings
 $(\sigma_{r_1+1},\sigma_{r_1+r_2+1})$, $\cdots$, $(\sigma_{r_1+r_2},\sigma_{r_1+2r_2})$, where $\sigma_{r_1+r_2+j}=\overline{\sigma_{r_1+j}}$ for $1\leq j\leq r_2$.
Then the degree of $K$ is $d=r_1+2r_2$.

Denote by $\cO_K$ the ring of integers in $K$ and by $U_K=\cO_K^*$ the group of units. Let $r=\rank(U_K)$ denote the unit rank of $K$, which equals $r_1+r_2-1$ by Dirichlet's Unit Theorem, in the rest of paper.

We denote hereafter $I=\{1,2,\cdots,r_1+r_2\}$, $d_i=1$ for $1\leq i\leq r_2$, and $d_i=2$ for $r_1+1\leq i\leq r_1+r_2$.

Recall that the algebraic norm of $\theta\in K$ is  $N_K(\theta)=\prod_{i=1}^d\sigma_i(\theta)$, in particular,
\begin{equation}\label{normdef}|N_K(\theta)|=\prod_{i=1}^d|\sigma_i(\theta)|=\prod_{i=1}^{r_1}|\sigma_i(\theta)|\cdot
\prod_{i=r_1+1}^{r_1+r_2}|\sigma_i(\theta)|^2=\prod_{i\in I}|\sigma_i(\theta)|^{d_i}.\end{equation}

$K$ can embedded into $\bR^{r_1}\oplus\bC^{r_2}$ by the canonical map
\begin{equation}\sigma:\theta\mapsto\big(\sigma_1(\theta),\cdots,\sigma_{r_1}(\theta),\sigma_{r_1+1}(\theta),\cdots,\sigma_{r_1+r_2}(\theta)\big),\end{equation}
and it is well known that $\sigma(\cO_K)$ is a cocompact lattice in $\bR^{r_1}\oplus\bC^{r_2}\cong\bR^d$. Therefore as $K=\cO_K\otimes_\bZ\bQ$, one can identify $K\otimes_\bQ\bR$
 with $\bR^{r_1}\oplus\bC^{r_2}$ via $\sigma$.

Actually, if $x=\theta\otimes s$ where $\theta\in K$, $s\in\bR$, then we can denote \begin{equation}\label{sigmaiKbar}\sigma_i(x)=s\sigma_i(\theta), \forall i=1,\cdots,d\end{equation} and again let \begin{equation}\label{sigmaKbar}\sigma(x)=\big(\sigma_1(x),\cdots,\sigma_{r_1}(x),\sigma_{r_1+1}(x),\cdots,\sigma_{r_1+r_2}(x)\big).\end{equation} Then $\sigma$ is an isomorphism between $K\otimes_\bQ\bR$ and $\bR^{r_1}\oplus\bC^{r_2}$.

Notice $K\subset K\otimes_\bQ\bR$ and that if $x\in K$ then the expressions (\ref{sigmaiKbar}) and (\ref{sigmaKbar}) agree with previous definitions.

Throughout this paper, we write \begin{equation}\Kb=K\otimes_\bQ\bR,\end{equation} identify it with $\bR^{r_1}\oplus\bC^{r_2}$ via the isomorphism $\sigma$, and equip it with the Euclidean distance from $\bR^{r_1}\oplus\bC^{r_2}$.

Each element $x\in\Kb$ will be represented as
\begin{equation}\label{Kbcoord}(x_i)_{i\in I}: x_i\in\bR\text{ if }1\leq i\leq r_1, x_i\in\bC\text{ if }r_1+1\leq i\leq r_1+r_2.\end{equation} And the distance between two points $x$ and $x'$ is $\|x-x'\|$ where \begin{equation}\|x\|=\big(\sum_{i\in I}|x_i|^2\big)^\frac12.\end{equation}

For all $R>0$, define a box-shaped compact subset \begin{equation}\label{box}B_R=\{x\in\Kb,|x_i|\leq R,\forall i\in I\}.\end{equation}

A point $x\in\Kb$ is said to be {\bf rational} if it is in $K$.

\begin{notation}In light of the identification above, from now on we will simply write $x_i$ for $\sigma_i(x)$ for all $x\in K$ and $i\in I$.\end{notation}

\begin{definition}\label{normKbdef} The algebraic norm on $\Kb$ is
\begin{equation}N_\Kb(x)=\prod_{i=1}^{r_1}x_i\prod_{i=r_1+1}^{r_1+r_2}|x_i|^2.\end{equation}
\end{definition}
Clearly $N_\Kb$ is a continuous function on $\Kb$ and satisfies~\eqref{inhomonorm}

Observe that $K$ acts naturally on $\Kb$ by multiplication: \begin{equation}\label{Kbarmulti}\theta'\cdot(\theta\otimes s)=\theta'\theta\otimes s, \forall\theta,\theta'\in K, \forall s\in\bR.\end{equation} In terms of the coordinate system (\ref{Kbcoord}), this multiplication writes: \begin{equation}\label{Kbarmulticoord}(\theta x)_i=\theta_ix_i,\forall i\in I.\end{equation}
It follows directly from definitions that
\begin{equation}\label{Normnatural}N_\Kb(xy)=N_K(x)N_\Kb(y), \forall x\in K,\forall y\in \Kb.\end{equation}

Define the logarithmic embedding map $\cL:U_K\mapsto\bR^I$ by \begin{equation}\label{Logemb}\cL(u)=\big(\log|u_i|\big)_{i\in I}.\end{equation} Then $\cL$ is a group morphism, and by Dirichlet's Unit Theorem its image $\cL(U_K)$ is a cocompact lattice in the subspace \begin{equation}\label{Logembspace}W=\big\{(w_i)_{i\in I}:\sum_{i\in I}d_iw_i=0\big\}.\end{equation}

The size of units from $U_K$ can be measured in terms of the following norm on $W$:
\begin{equation}\label{Wnorm}h_0(w)=\frac12\sum_{i\in I}d_i|w_i|=\sum_{\substack{i\in I\\w_i\geq 0}}d_iw_i=-\sum_{\substack{i\in I\\w_i\leq 0}}d_iw_i, \forall w\in W.\end{equation} The composition map $h^\Mah=h_0\circ\cL$ which maps $U_K$ to $[0,\infty)$ is called the {\bf logarithmic Mahler measure} on $U_K$. Then as $h_0(w)=0$  only if $w=0$, $h^\Mah(u)=0$ if and only if $u$ is a root of unity.

\begin{definition}$K$ is a {\bf CM}-number field if it satisfies one of the following equivalent conditions:

\begin{enumerate}[(i)]
\item $K$ is a totally complex quadratic extension of some totally real number field $F$;
\item There is a proper subfield $F$ such that $\rank(U_F)=\rank(U_K)$;
\end{enumerate}
\end{definition}

To see the conditions are actually equivalent, see for instance \cite{P75}.

\begin{remark}\label{CMconj}A CM number field $K$ has a natural complex conjugation $x\mapsto\bar x$ that is an automorphism and acts as the conjugation in $\bC$ in all embeddings of $K$. Moreover, the extension $K/F$ is normal and $\Gal(K/F)$ consists of the identity map and the complex conjugation (see \cite{W97}*{p39}).\end{remark}

\subsection{Relation with the group of units}

$m_\Kb(x)$ and $m_K(x)$ depend only on the residue class of $x$ modulo $\cO_K$. Hence we project them to the quotient $\Kb/\cO_K$:

\begin{definition}For $z\in\Kb/\cO_K$, let \begin{equation}m_{\Kb/\cO_K}(z)=\inf_{x\in\pi_{\cO_K}^{-1}(z)}|N_\Kb(x)|.\end{equation}\end{definition}

Here and in the sequel $\pi_\Gamma$ denotes the natural projection from $\Kb$ to $\Kb/\Gamma$ for any lattice $\Gamma\subset\Kb$. Obviously
\begin{equation}\label{mproj} m_\Kb=m_{\Kb/\cO_K}\circ\pi_{\cO_K}.\end{equation} Therefore \begin{equation}\label{KbSpec}\Spec(\Kb)=\{m_{\Kb/\cO_K}(z):z\in\Kb/\cO_K\};\end{equation} \begin{equation}\label{KSpec}\Spec(K)=\{m_{\Kb/\cO_K}(z):z\in K/\cO_K\};\end{equation} and by Proposition \ref{IEmin},
\begin{equation}M(K)=M(\Kb)=\sup_{z\in\Kb/\cO_K}m_{\Kb/\cO_K}(z)=\sup_{z\in K/\cO_K}m_{\Kb/\cO_K}(z).\end{equation} In particular, the function $m_{\Kb/\cO_K}$ is bounded by the expression (\ref{EMupper}).

$\Kb/\cO_K$ is a compact abelian group isomorphic to $\bT^d=\bR^d/\bZ^d$. Moreover, we equip $\Kb/\cO_K$ with the distance projected from $\Kb$ and denote it indifferently by $\|\cdot\|$, which makes $\Kb/\cO_K$ a locally Euclidean metric space. The volume of $\Kb/\cO_K$ is $\sqrt{D_K}$.

Moreover, suppose $z\in\Kb/\cO_K$ and $x\in\pi_{\cO_K}^{-1}(z)$, then $z$ is a torsion point if and only if $x$ is rational, in which case we say $z$ is rational as well.

Under the multiplicative action (\ref{Kbarmulti}), the group of units $U_K=\cO_K^*$ preserves the cocompact lattice $\cO_K$ in $\Kb$.

From now on let $G$ be a finite-index subgroup of $U_K$. The multiplication (\ref{Kbarmulti}) induces an action of $G$ on the compact quotient $\Kb/\cO_K$:
\begin{equation}\label{quotmulti}u.(x+\cO_K)=u.x+\cO_K, \forall u\in G, x\in\Kb.\end{equation}

The multiplication by $u$ on $\Kb/\cO_K$ is the identity map if and only if its lift on $\Kb$, which is given by (\ref{Kbarmulti}), is the identity. This happens only when $u=1$. Therefore the induced $G$-action is faithful.

Furthermore, the multiplication (\ref{quotmulti}) is continuous on $\Kb/\Gamma$ and preserves the additive structure, hence is actually an automorphism of the compact abelian group $\Kb/\cO_K$.

For all $z\in\Kb/\cO_K$ denote by $G.z=\{uz:u\in G\}$ the $G$-orbit of $z$ and by $\overline{G.z}$ the orbit closure.

\begin{lemma}\label{toralmin}Suppose $G$ is a finite-index subgroup of $U_K$ . Then,
\begin{enumerate}
\item For all $z,z'\in\Kb/\cO_K$, if $z'\in G.z$, then $m_{\Kb/\cO_K}(z')=m_{\Kb/\cO_K}(z)$;
\item The function $m_{\Kb/\cO_K}$ is upper semicontinuous on $\Kb/\cO_K$, that is, if $\lim_{n\rightarrow\infty}z_n=z$ then $\limsup_{n\rightarrow\infty}m_{\Kb/\cO_K}(z_n)\leq m_{\Kb/\cO_K}(z)$;
\item For all $z\in\Kb/\cO_K$, $$m_{\Kb/\cO_K}(z)=\min_{z'\in\overline{G.z}}m_{\Kb/\cO_K}(z')\leq \inf\{\|z'\|^d:z'\in\overline{G.z}\};$$
\item If $z$ in $K/\cO_K$, then $m_{\Kb/\cO_K}(z)$ lies in $\bQ$ and is attained by $|N_K(x)|$ at some $x\in\pi_{\cO_K}^{-1}(z)$.
\end{enumerate}
\end{lemma}

\begin{proof}{\noindent\it (1)} As $z'\in G.z\Leftrightarrow z\in G.z'$ it suffices to show $m_{\Kb/\cO_K}(z')\leq m_{\Kb/\cO_K}(z)$. Suppose $z'=uz$ where $u\in G\subset U_K$. Then for all $x\in\pi_{\cO_K}^{-1}(z)$,  $ux$ is in $\pi_{\cO_K}^{-1}(z')$ because $u\cO_K=\cO_K$. Thus $m_{\Kb/\cO_K}(z')\leq |N_\Kb(ux)|=|N_K(u)N_\Kb(x)|=|N_\Kb(x)|$ by (\ref{Normnatural}), where we used the fact that $N_K(u)=\pm 1$ when $u\in U_K$. Claim (1) is obtained by taking infimum over all $x$.

{\noindent\it (2)} Suppose the opposite is true. Then there must be a converging sequence $\{z_n\}$ with limit $z$ and a constant $\epsilon>0$ such that $m_{\Kb/\cO_K}(z_n)>m_{\Kb/\cO_K}(z)+\epsilon, \forall n$.
By the definition of $m_{\Kb/\cO_K}(z)$, one may find $x\in\pi_{\cO_K}^{-1}(z)$ such that $|N_\Kb(x)|<m_{\Kb/\cO_K}(z)+\epsilon$. There exist a sequence of lifts $\{x_n\}$ such that $\pi_{\cO_K}(x_n)=z_n, \forall n$ and $\lim_{n\rightarrow\infty}x_n=x$. Since $N_\Kb$ is a continuous function on $\Kb$, we have \begin{equation}\lim_{n\rightarrow\infty}m_{\Kb/\cO_K}(z_n)\leq\lim_{n\rightarrow\infty}|N_\Kb(x_n)|=|N_\Kb(x)|<m_{\Kb/\cO_K}(z)+\epsilon,\end{equation} contradicting the choice of $\{z_n\}$ and $\epsilon$.

{\noindent\it (3)} It follows from (1) and (2) that $m_{\Kb/\cO_K}(z)\leq\inf_{z'\in\overline{G.z}}m_{\Kb/\cO_K}(z')$. Because $z\in\overline{G.z}$, the equality holds and the infimum is actually a minimum. The second inequality follows from the fact that \begin{equation}|N_\Kb(x)|=\prod_{i\in I}|x_i|^{d_i}\leq\|x\|^d,\forall x\in\Kb\end{equation} and the definition of the metric $\|\cdot\|$ on $\Kb/\cO_K$.

{\noindent\it (4)} For $z\in K/\cO_K$, there is $q\in\bN$ such that $z$ is of order $q$ in $K/\cO_K$. Then for all $x\in\pi_{\cO_K}^{-1}(z)$, $qx\in\cO_K$. Hence $x\in q^{-1}\cO_K$ and $|N_K(x)|\in q^{-d}\bZ$. So $m_{\Kb/\cO_K}(z)$, which is the infimum of all the $|N_K(x)|$'s, lies in the discrete set $q^{-d}\bZ$ as well and equals $|N_K(x)|$ for at least one $x\in\pi_{\cO_K}^{-1}(z)$.\end{proof}

\subsection{Reduction to a bounded domain}\label{subsecbounded}
We introduce an upper bound on the size of fundamental units by

\begin{equation}\cF_{U_K}=\min_{(u_1,\cdots, u_r)}\max_{l=1}^rh^\Mah(u_l),\end{equation} where the minimum is taken over all sets of fundamental units $(u_1,\cdots,u_r).$
Recall $u_1,\cdots,u_r\in U_K$ form a set of fundamental units if: (1) they are multiplicatively independent, i.e. $\prod_{l=1}^ru_l^{e_l}, e_l\in\bZ$ is equal to $1$ if and only if the $e_l$'s are all zero, and (2) together with all roots of unity in $K$, they generate $U_K$. $\cF_{U_K}$ is well defined and strictly positive.

Take the previously defined convex norm $h_0$ on $W$ and recall that $\cL(U_K)$ is a cocompact lattice of the $r$-dimensional vector space $W$. Consider the successive minima $0<t_1\leq t_2\cdots\leq t_r$ of $\cL(U_K)$ with respect to $h_0$. Then it follows from the fact $h^\Mah=h_0\circ\cL$ and the definition of $\cF_{U_K}$ that $\cF_{U_K}=t_r$.

Recall that the {\bf regulator} $R_K$ of $K$ is the determinant of the lattice $\cL(U_K)$. So Minkowski's theorem implies that $R_K\ll_d\prod_{j=1}^rt_j$. But each $t_j$ can be written as $h_0(\cL(u))=h^\Mah(u)$ for some $u\in U_K$ of infinite order. And it is known that the logarithmic Mahler measure of any algebraic unit of infinite order is bounded from below by a constant depending only on its algebraic degree (see for example \cite{V96}), i.e. $t_j\gg_d 1$ for all $j$. Hence \begin{equation}1\ll_d\cF_{U_K}\ll_d(t_1\cdots t_{r-1})^{-1}R_K\ll_dR_K.\end{equation} Moreover, by the work of Sands \cite{S91}, $R_K\ll_{d}D_K^{\frac12}(\log D_K)^d$. Therefore,
\begin{equation}\label{heightbd}1\ll_d\cF_{U_K}\ll_dD_K^{\frac12}(\log D_K)^d.\end{equation}

\begin{lemma}\label{compactness}Suppose $G$ is a finite-index subgroup of $U_K$, then there is a constant $C$ depending on $K$ and $G$ such that for any non-zero element $x\in\Kb$, there is $g\in G$ satisfying that $|(gx)_i|\leq C|N_\Kb(x)|^{\frac1d}$ for all $i\in I$.

Moreover, when $G=U_K$, $C$ can be taken to be $e^{\frac12r\cF_{U_K}}$.\end{lemma}

\begin{proof} We give first a proof in the special case that $G=U_K$.

Define $w\in\bR^I$ by $w_i=\log|x_i|-\frac1d\sum_{j\in I}d_j\log|x_j|$. Then $\sum_{j\in I}d_jw_j=0$ and thus $w$ is in the space $W$ given by (\ref{Logembspace}).

As above, let $t_1,\cdots,t_r$ be the successive minima of $\cL(U_K)$ with respect to $h_0$. By a theorem of Jarn\'{i}k (see \cite{GL87}*{p99}), there is a vector $y\in\cL(U_K)$ such that $h_0(w-y)\leq\frac{\sum_{j=1}^{r}t_j}2\leq\frac 12r\cF_{U_K}$. In particular, $w_i-y_i\leq d_i|w_i-y_i|\leq h_0(w-y)\leq\frac12r\cF_{U_K}$ for all $i\in I$ by (\ref{Wnorm}).

Suppose $y=\cL(v)$ where $v\in U_K$, then $w_i-y_i$ is just $\log|x_i|-\log|v_i|-\frac1d\sum_{j\in I}d_j\log|x_j|$, which equals $\log|(v^{-1})_ix_i|-\frac1d\log|N_\Kb(x)|$. Hence the previous inequality says $|(v^{-1})_ix_i|\leq e^{\frac12r\cF_{U_K}}|N_\Kb(x)|^{\frac1d}$. This proves the lemma form $G=U_K$.

From this special case one can easily deduce the general statement for any finite-index subgroup $G\subset U_K$. Actually,  Fix a set $A$ consisting of one representative from each residue class in the finite quotient $U_K/G$. We already proved there is $u\in U_K$ such that $|(ux)_i|<e^{\frac12r\cF_{U_K}}|N_\Kb(x)|^{\frac1d}$. Pick $g\in G$ such that $g^{-1}u\in A$. Then $|(gx)_i|<C|N_\Kb(x)|^{\frac1d}$ where $C=(\max_{\substack{a\in A\\j\in I}}|a_j|)\cdot e^{\frac12r\cF_{U_K}}$.\end{proof}

\begin{corollary}\label{finiteres}Suppose $G$ is a finite-index subgroup of $U_K$ then there exists a constant $R>0$ such that \begin{equation}\label{finitereseq}m_\Kb(x)=\inf_{x'\in (G.x+\cO_K)\cap B_R}|N_\Kb(x')|,\forall x\in\Kb\end{equation} where $B_R$ is defined as in (\ref{box}).

Moreover, one can take $R=D_K^{\frac1d}e^{\frac12r\cF_{U_K}}$ if $G=U_K$.\end{corollary}

\begin{proof}For $z=\pi_{\cO_K}(x)$, it follows from Lemma \ref{toralmin} that \begin{equation}m_\Kb(x)= m_{\Kb/\cO_K}(z)
= \inf_{z'\in G.z}m_{\Kb/\cO_K}(z)=\inf_{x'\in G.x+\cO_K}|N_\Kb(x')|.\end{equation} This shows $m_\Kb(x)$ is bounded from above by the right-hand side of (\ref{finitereseq}).

On the other hand, for an arbitrarily small $\epsilon>0$, one can pick $y\in x+\cO_K$ such that $|N_\Kb(y)|\leq m_\Kb(x)+\epsilon$. Since $m_\Kb(x)\leq M(\Kb)\leq 2^{-d}D_K$ by (\ref{EMupper}), we may assume $|N_\Kb(y)|\leq D_K$. Lemma \ref{compactness} asserts that there exists an element $x'$ in \begin{equation}G.y\cap B_{C|N_\Kb(y)|^{\frac1d}}\subset (G.x+\cO_K)\cap B_{C|N_\Kb(y)|^{\frac1d}}\subset  (G.x+\cO_K)\cap B_R.\end{equation} where $C$ is the constant in the lemma and $R=CD_K^\frac1d$. Then $|N_\Kb(x')|=|N_\Kb(y)|\leq m_\Kb(x)+\epsilon$ by the $G$-invariance in Lemma \ref{toralmin}, and thus as $\epsilon$ can be arbitrarily small we see $m_\Kb(x)\geq\inf_{x'\in (G.x+\cO_K)\cap B_R}|N_\Kb(x')|$, which completes the proof.
\end{proof}

\section{Computational complexity in non-CM fields}\label{seccomplexity}

In this section, let $K$ be a non-CM field whose unit rank is $r\geq 2$.

\subsection{Rigidity of Cartan actions by toral automorphisms}

In light of Lemma \ref{toralmin}, in order to understand $m_{\Kb/\cO_K}(z)$ it may be helpful to study the $G$-orbit of $z$ in the torus $\Kb/\cO_K$.

\begin{definition}A toral automorphism $\varphi\in\Aut(\bT^d)=\GL(d,\bZ)$ is {\bf irreducible} if $\bT^d$ has no proper non-trivial $\varphi$-invariant subtorus. $\varphi$ is {\bf totally irreducible} if $\varphi^k$ is irreducible for all non-zero integers $k$.\end{definition}

$\varphi$ is irreducible if and if its characteristic polynomial is irreducible over $\bQ$, or equivalently the eigenvalues of $\varphi$ are distinct conjugate algebraic numbers of degree $d$.

As we have seen in (\ref{Kbarmulticoord}), when one identifies $\Kb$ with $\bR^{r_1}\oplus\bC^{r_2}$ by $\sigma$, the multiplication by $\theta\in K$ multiplies on the $i$-copy of $\bR$ by $\theta_i$, and on the $j$-th copy of $\bC$ by $\theta_{r_1+j}$ or, if we view $\bC$ as $\bR^2$, by the matrix $\left(\begin{array}{cc}\re{\theta_{r_1+1}}&-\im{\theta_{r_1+1}}\\
\im{\theta_{r_1+1}}&\re{\theta_{r_1+1}}\end{array}\right),$ which have eigenvalues $\theta_{r_1+j}$ and $\theta_{r_1+r_2+j}$. Hence with respect to some complex eigenbasis (of $\Kb\otimes_\bR\bC=K\otimes_\bQ\bC$), the multiplication by $\theta$, which is a linear transformation on $\Kb$, can be diagonalized as $\diag{\theta_1,\cdots,\theta_d}$ simultaneously for all $\theta\in K$. It follows that:

\begin{remark}\label{irrrmk} Each $u\in U_K$ acts as an irreducible toral automorphism on $\Kb/\cO_K$ if and only if $\bQ(u)=K$.\end{remark}

\begin{definition}
Suppose $X\cong\bT^d$ and $G\subset\Aut(X)$ is an abelian subgroup of toral automorphisms on $X$. We say the action $G\curvearrowright X$ is {\bf Cartan} if the following conditions are satisfied:\begin{enumerate}[(i)]
\item There is an element $g\in G$ which acts as an irreducible toral automorphism of $X$;
\item One cannot find a larger abelian subgroup $G_1\supset G$ in $\Aut(X)$, such that $\rank(G_1)>\rank(G)$.
\end{enumerate}
\end{definition}

\begin{theorem}\label{Berend}{\bf (Berend, \cite{B83}) }
Suppose $G\curvearrowright X$ is a faithful Cartan action by automorphisms on a torus $X$, where $G$ is an abelian group of rank $r\geq 2$ and at least one element $g\in G$ acts as a totally irreducible toral automorphism. Then every $G$-orbit is either a finite set of torsion points, or dense in $X$.
\end{theorem}
\begin{proof}Berend's original theorem from \cite{B83} was in fact much stronger as he showed the conclusion above holds for a much larger class of semigroup actions by commuting toral endomorphisms. It was known that the assumptions stated here imply Berend's conditions when $G\cong\bZ^r$ (see, for instance, \cite{LW10}*{\S 2.4}).  For a general finitely generated abelian group $G$, there is a finite subgroup $G'$ that is isomorphic to $\bZ^r$ and the rigidity of the $G$-action on $X$ follows easily from that of its restriction to $G'$.\end{proof}

We now make the link between the $U_K$-action on $\Kb/\cO_K$ and dynamics.

\begin{lemma}\label{UKCartan}Suppose $K$ is a non-CM number field of unit rank $r\geq 2$, and $G\subset U_K$ is a finite-index group that is isomorphic to $\bZ^r$. Then:
\begin{enumerate}[(1)]
\item There is $u\in G$ that acts as a totally irreducible automorphism on $\Kb/\cO_K$.
\item The multiplicative $G$-action (\ref{Kbarmulti}) on $\Kb/\cO_K$ is Cartan.
\end{enumerate}
\end{lemma}

\begin{proof} (1) By Remark \ref{irrrmk}, it suffices to find $u\in G$ such that for each non-zero integer $k$, $\bQ(u^k)=K$. It can be checked that there is  $u_0\in U_K$ that has this property (see \cite{C06}*{Lemma 2}). As $G$ has finite index in $U_K$, it contains a non-trivial power $u$ of $u_0$, which has the desired property as well.

(2) Suppose the $G$-action is not Cartan, that is, there is a larger abelian group $G_1\supset G$ that acts on $\Kb/\cO_K$ by toral automorphisms, such that $\rank(G_1)>\rank(G)$ and the restriction of the $G_1$-action to $G$ coincide with the previously defined multiplicative action.

As by part (1) $G$ contains an element $u$ such that $\bQ(u)=K$. Then $1,u,\cdots,u^{d-1}$ are linearly independent over $\bQ$ where $d$ denotes the degree of $K$.

Take an arbitrary element $A$ from $G_1$. Then $A$ can be regarded as an element from $\Aut(\Kb/\cO_K)$, or equivalently, an element from $\GL(\Kb)$ that preserves the lattice $\cO_K$. Consider the element $1$ from $\cO_K$, then $A.1\in\cO_K$ and we denote it by $\gamma$. Hence $(A-\gamma).1=0$ where $A$ and $\gamma$ are both regarded as linear maps from $\Kb$ to itself. Since $G_1$ is abelian, $A$ commutes with the $G$-action by definition. In particular, $A$ commutes with the multiplication by $u$ as elements from $GL(\Kb)$. Hence for any power $u^k$, \begin{equation}(A-\gamma).u^k=(A-\gamma).(u^k.1)=u^k.(A-\gamma).1=0,\end{equation} where $u^k$ is regarded as a vector from $\Kb$ in the first expression and as a linear map thereafter. However since $1,u,\cdots,u^{k-1}$ are linearly independent, they span the $\bQ$-vector space $K$ and hence $\Kb=K\otimes_\bQ\bR$ as well. It follows that $(A-\gamma).v=0$ for any $v\in\Kb$. In other words, as an element from $\GL(\Kb)$, $A$ coincide with the multiplication by $\gamma\in\cO_K$.

Apply the same argument again, we wee $A^{-1}$ coincide with the multiplication by some $\beta\in\cO_K$. Then the multiplicative action by $\gamma\beta\in\cO_K$ on $\Kb$ is trivial, which is possible only if $\gamma\beta=1$. Hence $\gamma$ belongs to $U_K=\cO_K^*$; that is, $A$ is actually the multiplication by some element from $U_K$.

Since this is true for all $A\in G_1$. $G_1$ can be regarded as a subgroup of $U_K$. So as $G$ has finite index in $U_K$, $G_1$ cannot have strictly higher rank than $G$, which contradicts the assumption. This completes the proof.\end{proof}

\subsection{Effective aspects of rigidity}

In order to show Theorem \ref{complexity}, we will apply Proposition \ref{effBerend} below, an effective version of Theorem \ref{Berend}, to the action $U_K\curvearrowright\Kb/\cO_K$. Before doing this, one needs to introduce the notion of distortion.

The distortion of an isomorphism $\psi:\bR^d\mapsto\Kb$ can be measured by

\begin{equation}\label{distortion}\cM_\psi=\max\Big(\|\psi\|\Vol\big(\Kb/\psi(\bZ^d)\big)^{-\frac 1d},\|\psi^{-1}\|\Vol\big(\Kb/\psi(\bZ^d)\big)^\frac1d\Big),\end{equation}
where $\|\psi\|$ and $\|\psi^{-1}\|$ are norms of linear maps.  $\cM_\psi$ is always greater than or equal to $1$.

Note that $\psi$ projects to an isomorphism between $\bT^d$ and $\Kb/\psi(\bZ^d)$, which we denote by $\psi$ as well.

The distortion of an ideal lattice in $\Kb$ can be bounded in terms of the discriminant:

\begin{lemma}\label{idealreg}\cite{W10}*{Lemma A.5} If $I$ is an ideal in $\cO_K$, then there exists an isomorphism $\psi$ between $\bR^d$ and $\Kb$ such that $\psi(\bZ^d)=I$ and $\cM_\psi=O_d\big(D_K^{\frac{d-1}{2d}}\big)$
\end{lemma}

\begin{proposition}\label{effBerend}\cite{W10}*{Proposition 7.6}
\footnote{In \cite{W10}*{Proposition 7.6}, the density parameter is misstated as $(\log\log\log q)^{-C^{-1}\cF_{U_K}^2}$ instead of $(\log\log\log q)^{-C^{-1}\cF_{U_K}^{-2}}$; which results from mistakenly copying the exponent from Proposition 7.1. The proposition is otherwise not affected.}
Let $K$ be a degree $d$ non-CM number field of unit rank $2$ or higher, $\Gamma\subset\cO_K$ be a full rank sublattice preserved by $U_K$ under multiplication, $\psi$ be an isomorphism from $\bR^d$ to $\Kb$ such that $\psi(\bZ^d)=\Gamma$, and $q$ be a positive integer greater than or equal to $\exp\exp\exp\max\big(C\cM_\psi^{30d},\max(\cF_{U_K},2)^{C\cF_{U_K}^2}\big)$. Then for any torsion element $z$ in $\Kb/\Gamma$ of order at least $q$, the preimage $\psi^{-1}(U_K.z)$ of the orbit $U_K.z\subset\Kb/\Gamma$ is $(\log\log\log q)^{-C^{-1}\cF_{U_K}^{-2}}$-dense in $\bT^d$. Here $C>1$ is an effective constant that depends only on $d$.
\end{proposition}

\begin{corollary}\label{denombd} Let $K$ be as in Proposition \ref{effBerend}. There is a constant $C>1$ that depends only on $d$ such that if
\begin{equation}Q=\exp\exp\exp(D_K^{C\cF^2_{U_K}}),\end{equation}
then:\begin{enumerate}
\item  $m_{\Kb/\cO_K}(z)<2^{-d}$ for any rational element $z\in\Kb/\cO_K$ whose minimal order is greater than or equal to $Q$.
\item For $R=D_K^{\frac1d}e^{\frac12r\cF_{U_K}}$, \begin{equation}M(K)=\sup_{x\in
\bigcup_{\substack{q\in\bN\\2\leq q<Q}}q^{-1}\cO_K}\left(\min_{x'\in (U_K.x+\cO_K)\cap B_R}|N_K(x')|\right).\end{equation}
\end{enumerate}
\end{corollary}

\begin{proof}(1) By Lemma \ref{idealreg}, there is an isomorphism $\psi$ from $\bR^d$ to $\Kb$ that sends $\bZ^d$ to $\cO_K$, such that $\cM_\psi=O_d(D_K^{\frac{d-1}{2d}})$. Then by (\ref{distortion}), \begin{equation}\|\psi\|\leq\cM_\psi\cdot\Vol(\Kb/\cO_K)^\frac1d=\cM_\psi D_K^\frac1{2d}=O_d(D_K^{\frac12}).\end{equation}

Following this estimate and the inequality (\ref{heightbd}), when $C$ is large enough, \begin{equation}D_K^{C\cF_{U_K}^2}\geq\max\big(C_0\cM_\psi^{30d},\max(\cF_{U_K},2)^{C_0\cF_{U_K}^2}\big)\end{equation} where $C_0$ is the constant from Lemma \ref{effBerend}.

Therefore if $Q\geq\exp\exp\exp(D_K^{C\cF^2_{U_K}})$ and $z$ is as in the statement, then it follows from the lemma that there exists $u\in U_K$ such that the preimage $\psi^{-1}(uz)$ lies within distance less than $(\log\log\log Q)^{-C_0^{-1}\cF_{U_K}^{-2}}$ from the origin in $\bT^d$. So \begin{equation}\|uz\|<\|\psi\|(\log\log\log Q)^{-C_0^{-1}\cF_{U_K}^{-2}}
\ll_d D_K^{\frac12} D_K^{-\frac{C}{C_0}}.\end{equation} In particular, by Lemma \ref{toralmin}.(3),
\begin{equation}m_{\Kb/\cO_K}(z)<\big(\cM_\psi D_K^{-\frac{C}{C_0}}\big)^d=O_d\big((D_K^{-\frac{C}{C_0}+\frac12})^d\big).
\end{equation}
Because $D_K\geq 2$, if $C$ is sufficiently large (depending only on $d$) then $m_{\Kb/\cO_K}(z)<2^{-d}$.

(2) Notice first that $m_K(\frac12)$ clearly doesn't vanish and belongs to $2^{-d}\bZ$ by the proof of Lemma \ref{toralmin}.(4). Thus $M(K)\geq m_K(\frac12)\geq 2^{-d}$.

Then for all $x$ with denominator  greater than or equal to $Q$, its projection $z$ in $\Kb/\cO_K$ is a rational point as required by (1). Thus $m_K(x)=m_\Kb(x)=m_{\Kb/\cO_K}(z)< 2^{-d}\leq M(K)$. It follows from this fact and definition that $M(K)$ is the supremum of $m_K(x)$ where $x$ has a small denominator:
\begin{equation}M(K)=\sup\{m_K(x):x\in\bigcup_{q\in\bN, 2\leq q<Q} q^{-1}\cO_K\}.\end{equation}
$q$ is supposed to be at least $2$ since the $q=1$ case is not interesting where $x\in\cO_K$ and $m_K(x)=0$.
It suffices to apply Corollary \ref{finiteres} to complete the proof.\end{proof}

\begin{proof}[Proof of Theorem \ref{complexity}] By Corollary \ref{denombd}, to determine $M(K)$ it suffices to calculate and compare $|N_K(x')|$ for all $x'$ from the union $A$ of those $\Omega_x$'s where $x\in\bigcup_{q\in\bN, q<Q} q^{-1}\cO_K$ and $\Omega_x=(U_K.x+\cO_K)\cap B_R$ with $R=D_K^{\frac1d}e^{\frac12r\cF_{U_K}}$. Notice the $\Omega_x$'s are all finite, and are either equal or disjoint for different $x$'s. So $A$ can be regarded as a disjoint union and every $|N_K(x')|$ comes only once into the comparison.

Since every element from $U_K.x+\cO_K$ has the same denominator as $x$, $A\subset\bigcup_{\substack{q\in\bN\\ q<Q}} q^{-1}\cO_K$. So $A=\bigcup_{\substack{q\in\bN\\ q<Q}}A^{(q)} $ where $A^{(q)}=A\cap q^{-1}\cO_K$.

Moreover, $A$ is clearly inside the box $B_R$. So $A^{(q)}\subset q^{-1}\cO_K\cap B_R$.

For two distinct elements $x',y'\in A^{(q)}$, $x'-y'\in q^{-1}\cO_K$. So one can deduce $|N_K(x'-y')|\geq q^{-d}$ and it follows that for at least one $i\in I$ we have $|(x'-y')_i|\geq q^{-1}$. Therefore $\bigcup_{x'\in A^{(q)}}(x'+B_{\frac1{2q}})$ is a disjoint union. Furthermore, this union is covered by $B_{R+\frac1{2q}}$.  Hence one an easily see that \begin{equation}|A^{(q)}|\leq\frac{\Vol(B_{R+\frac1{2q}})}{\Vol(B_{\frac1{2q}})}=\left(\dfrac{R+\frac1{2q}}{\frac1{2q}}\right)^d=(2qD_K^{\frac1d}e^{\frac12r\cF_{U_K}}+1)^d.\end{equation}

Because $D_K>1$ and $q\in\bN$, $qD_K^{\frac1d}e^{\frac12r\cF_{U_K}}>1$. Therefore

\begin{equation}\begin{split}|A|\leq&\sum_{q\in\bN, q<Q}(3qD_K^{\frac1d}e^{\frac12r\cF_{U_K}})^d\leq Q^{d+1}(3D_K^{\frac1d}e^{\frac12r\cF_{U_K}})^d\\
\leq &\exp\Big((d+1)\exp\exp(D_K^{C_0\cF^2_{U_K}})\Big)(3D_K^{\frac1d}e^{\frac12r\cF_{U_K}})^d\\
\leq &\exp\exp\exp(D_K^{C_1\cF^2_{U_K}}),
\end{split}\end{equation}
where $C_0$ is the constant denoted by $C$ in Corollary \ref{denombd} and $C_1$ is a larger constant, chosen in a way that depends only on $d$, which is possible because $D_K\geq 4$, $\cF_{U_K}\gg_d 1$ and $r\leq d-1$.

Therefore, one only needs to compute and compare the algebraic norms of at most $\exp\exp\exp(D_K^{C_1\cF^2_{U_K}})$ numbers from $A\subset K$. For each $x\in A$, its denominator is bounded by $Q$ and all its archimedean embeddings are bounded by $R=D_K^{\frac1d}e^{\frac12r\cF_{U_K}}$. It is known (see e.g. \cite{B04}) that the complexity of computing the algebraic norm of such a number is polynomial in $\log Q+\log R$, which is $O_d(\exp\exp(D_K^{(C_0+\epsilon)\cF^2_{U_K}})$ for any $\epsilon>0$.  Other operations needed in the computation only require relatively cheap costs, for example a set of fundamental units can be determined with complexity $O_d(D_k^{\frac14})$ by \cite{FJ10}. Hence the total complexity of computing $M(K)$ is bounded by $\exp\exp\exp(D_K^{C\cF^2_{U_K}})$ for a constant $C$ slightly larger than $C_1$.
\end{proof}

\section{Euclidean Spectra of CM fields}

The dynamics become very different for CM fields, in whose cases the action by $U_K$ on $\Kb/\cO_K$ is essentially not irreducible any more. To see this, let $F$ be a maximal totally real subfield of $K$, that is, a totally real subfield over which $K$ is a totally complex quadratic extension. Recall that $U_F$ is a finite index subgroup of $U_K$. $F$ is a $U_F$-invariant subspace of the $\bQ$-vector space $K$ and the $s$-dimensional real subspace $\Fb=F\otimes_\bQ\bR\subset\Kb$ projects to a subtorus $T$ of $\Kb/\cO_K$ that is invariant under the multiplicative action by $U_F$. So the action by any element of $U_F$ on $\Kb/\cO_K$ cannot be an irreducible toral automorphism. Since $U_F$ is of finite index, it follows that no element of $U_K$ acts totally irreducibly. Hence the $U_K$ action on $\Kb/\cO_K$ doesn't satisfy the total irreducibility condition in Theorem \ref{Berend}. In consequence, orbits of irrational points or of rational points with large denominators don't have to be close to the origin (see also \cite{C06}*{Remark 3}).

In this section, let $K$ be a CM field, $F$ be the associated maximal totally real subfield and $\Fb=F\otimes_\bQ\bR$. Then $r_1=0$, $d=2r_2$. For simplicity denote $s=r_2$, then $\deg F=s$ and $\deg K=d=2s$. $K$ has $s$ pairs of imaginary embeddings $(\sigma_1,\sigma_{s+1})$, $\cdots$, $(\sigma_s,\sigma_{2s})$ and $F$ has $s$ real embeddings $\tau_1,\cdots,\tau_s$. For all $i\in I=\{1,\cdots,s\}$, the restrictions of $\sigma_i$ and $\sigma_{s+i}$ to $F$ both coincide with $\tau_i$. Moreover, both $U_F$ and $U_K$ have rank $r=s-1$.

\subsection{Product structure of $\Kb$}

We hope to follow the same strategy as before by looking at the action on $\Kb/\cO_K$ by some subgroup of $U_K$. As $U_F$ coincide with $U_K$ up to finite index, it would be helpful if $\Kb/\cO_K$ has a product structure with respect to the $s$-dimensional subtorus $\Fb/\cO_F$. However, the existence of such a product structure is not clear and therefore instead of $\Kb/\cO_K$ we will work on a finite cover of it which splits as a product.

Fix an element $\eta\in\cO_K$ such that $\eta\notin\cO_F$. Then $K=F\oplus\eta F$ and $\Kb=\Fb\oplus\eta\Fb$. Define a finite-index sublattice in $\cO_K$ by $\Gamma=\cO_F\oplus\eta\cO_F$. Clearly $\Gamma$ is invariant under multiplication by elements of $U_F$. Hence $U_F$ naturally acts on $\Kb/\Gamma$.

$\Kb/\Gamma$ is isomorphic to $(\Fb/\cO_F)^2$. Actually, there is a unique isomorphism $\rho$ that sends each $x\in\Kb$ to $\rho(x)=\big(\rho^{(1)}(x),\rho^{(2)}(x)\big)\in \Fb^2$ in such a way that \begin{equation}\label{etadecomp}x=\rho^{(1)}(x)+\eta\rho^{(2)}(x).\end{equation} In addition, $x$ belongs to $\Gamma$ if and only if $\rho(x)\in \cO_F^2$. Therefore the map $\rho$ induces a continuous isomorphism between $\Kb/\Gamma$ and $(\Fb/\cO_F)^2$, which is denoted indifferently by $\rho=(\rho^{(1)},\rho^{(2)})$.

Since $\Gamma\subset\cO_K$ is of finite index, $\Kb/\Gamma$ is a finite cover of $\Kb/\cO_K$. By writing $\pi_{\Gamma,\cO_K}$ for the corresponding finite-to-one projection, $\pi_{\cO_F^2}$ for the natural projection from $\Fb^2$ to $(\Fb/\cO_F)^2$, and $\pi_\triangle$ for $\pi_{\Gamma,\cO_K}\circ\rho^{-1}$, we complete the following commutative diagram:

\newcommand{\curvedarrow}[1]{%
\setlength{\unitlength}{0.03\DiagramCellWidth}
\begin{picture}(0,0)(0,0)
\qbezier(20,-50)(-18,7)(20,64)
\qbezier(19,-42)(19.5,-46)(20,-50)
\qbezier(14,-46)(17,-48)(20,-50)
\put(-8,8){\makebox(0,0)[t]{$\scriptstyle {#1}$}}
\end{picture}
}

\begin{equation}\label{commdiag}
\begin{diagram}
&\Kb &\rTo^\sim_\rho &\Fb^2\\
&\dTo_{\pi_\Gamma} & &\dTo_{\pi_{\cO_F^2}}\\
\curvedarrow{\pi_{\cO_K}}&\Kb/\Gamma &\rTo^{\sim\ \ \ }_{\rho\ \ \ } &(\Fb/\cO_F)^2\\
&\dTo_{\pi_{\Gamma,\cO_K}} &\ldTo_{\pi_\triangle}&\\
&\Kb/\cO_K &&\\
\end{diagram}
\end{equation}

As $\rho$ is an isomorphism, $\pi_\triangle$ is also a finite covering map.

A point in one of the spaces in the diagram is said to be {\bf rational} if it either sits in $K\subset\Kb$ or $F^2\subset\Fb^2$, or descends from such a point. If a point is rational, then so are all its images and preimages in the diagram. In the tori $\Kb/\Gamma$, $(\Fb/\cO_F)^2$, and $\Kb/\cO_K$, rational points are exactly the torsion points.

Without causing ambiguity, subscript $i$ will indicate the $i$-th coordinates in both $\Kb$ and $\Fb$, which correspond respectively to the embeddings $\sigma_i$ of $K$ and $\tau_i$ of $F$. Notice $x_i$ is complex for $x\in\Kb$ but $y_i$ is real for $y\in\Fb$.

The decomposition (\ref{etadecomp}) can be expressed easily in terms of the coordinates:\begin{equation}\label{coorddecomp}x_i=\big(\rho^{(1)}(x)\big)_i+\eta_i\big(\rho^{(2)}(x)\big)_i, \forall x\in\Kb, \forall i\in I.\end{equation} Note $\eta_i\in\bC$ but $(\rho^{(1)}(x)\big)_i,\big(\rho^{(2)}(x)\big)_i\in\bR$. Furthermore, because $\rho$ is an isomorphism, we must have \begin{equation}\label{etaimag}\eta_i\notin\bR, \forall i\in I.\end{equation}

We rewrite the norm in $\Kb$:
\begin{equation}\label{normdecomp}\begin{split}N_\Kb(x)=&\prod_{i=1}^s|x_i|^2=\prod_{i=1}^s\Big|\big(\rho^{(1)}(x)\big)_i+\eta_i\big(\rho^{(1)}(x)\big)_i\Big|^2= N_*\big(\rho(x)\big)\end{split}\end{equation}

Where $N_*$ is the functional on $\Fb^2\cong(\bR^s)^2$ defined by
\begin{equation}\label{rhonorm}\begin{split}
N_*(y^{(1)},y^{(2)})=&\prod_{i=1}^s\left(\big(y^{(1)}_i+(\re{\eta_i})y^{(2)}_i\big)^2+(\im{\eta_i})^2(y^{(2)}_i)^2\right).
\end{split}\end{equation}
By (\ref{etaimag}), each factor in the product is a positive definite quadratic polynomial in $y^{(1)}_i$ and $y^{(2)}_i$. In particular, $N_\Kb$ and $N_*$ are always non-negative.

Let $U_F$ act diagonally both on $\Fb^2$ and on $(\Fb/\cO_F)^2$: given $u\in U_F$, for $y=(y^{(1)},y^{(2)})\in\Fb^2$, $uy$ will stand for $(uy^{(1)},uy^{(2)})$ and similarly on $(\Fb/\cO_F)^2$.

\begin{remark}\label{CDaction}We have made $U_F$ act on all the spaces in the Diagram (\ref{commdiag}). Since all the actions descend from the multiplicative action (\ref{Kbarmulti}) on $\Kb$, the $U_F$-actions commute with the maps in the diagram.\end{remark}

As $F$ is totally real, it is in particular not CM. Therefore by Lemma \ref{UKCartan}, the $U_K$-action on $\Fb/\cO_F$ is Cartan and contains a totally irreducible element. The same are true for the restriction of the action to any finite-index subgroup $G\subset U_F$.

\subsection{Rigidity of the diagonal action}

Note for each $\varphi\in F$, the subset \begin{equation}V^\varphi=\{(y^{(1)},y^{(2)})\in\Fb^2:y^{(1)}=\varphi y^{(2)}\},\end{equation} is an $s$-dimensional subspace of $\Fb^2$. For $\varphi=\infty$, let \begin{equation}V^\infty=\{(y^{(1)},y^{(2)})\in\Fb^2:y^{(2)}=0\}.\end{equation}

\begin{definition}
An {\bf $s$-dimensional homogeneous $G$-invariant subtorus} of $(\Fb/\cO_F)^2$ is the projection of $V^\varphi$ to $(\Fb/\cO_F)^2$ for some $\varphi\in F\cup\{\infty\}$, denoted by $T^\varphi$.
\end{definition}

It is not hard to check $T^\varphi$ is indeed an $s$-dimensional subtorus (\cite{LW10}*{Lemma 3.3}). Moreover, $V^\varphi$ is clearly invariant under the action by $U_F$, and hence so is $T^\varphi$.

Fix from now on a subgroup $G\subset U_F$ such that $G\cong\bZ^r$. Then $G$ is of finite index in $U_F$, and thus in $U_K$ as well. Restrict the action $U_F\curvearrowright(\Fb/\cO_F)^2$ to $G$.

\begin{definition}
An {\bf $s$-dimensional homogeneous $G$-invariant subset} of $(\Fb/\cO_F)^2$ is a subset of the form $G.(T^\varphi+\goz)=\{g.(\goz'+\goz):g\in G,\goz'\in T^\varphi\}$ for some fixed $\varphi\in F\cup\{\infty\}$ and rational point $\goz\in (F/\cO_F)^2$. We call $\varphi$ the {\bf slope} of $G.(T^\varphi+\goz)$.

An $s$-dimensional homogeneous $G$-invariant subset of slope $\varphi$ in $\Kb/\cO_K$ is the projection of such a subset in $(\Fb/\cO_F)^2$ by $\pi_\triangle$.

\end{definition}

Clearly every $s$-dimensional homogeneous $G$-invariant subset is indeed invariant under the $G$-action.

We list a few basic properties of $s$-dimensional homogeneous $G$-invariant subsets :

\begin{lemma}\label{homofacts}  Let $K$, $F$, $\Gamma$ and $G$ be as above. Suppose the unit rank of $K$ and $F$ is $r\geq 2$, then the following claims hold in both $(\Fb/\cO_F)^2$ and $\Kb/\cO_K$: \begin{enumerate}[(1)]
\item Suppose $L$ is an $s$-dimensional homogeneous $G$-invariant subset of slope $\varphi\in F\cup\{\infty\}$ where $s=\deg F=r+1$. Then $L$ is a finite disjoint union $\bigcup_{k=1}^h T_k$ where each $T_k$ is a translate of the subtorus $T^\varphi$ (resp. the subtorus $\pi_\triangle(T^\varphi)$ in $\Kb/\cO_K$) by a rational point;
\item Let $L$ be as in (1). For a point in $L$, its $G$-orbit is either finite or dense in $L$ depending on whether the point is rational or not;
\item Any infinite $G$-invariant closed subset contains at least one $s$-dimensional homogeneous $G$-invariant subset;
\item Given $\epsilon>0$, there are only finitely many $s$-dimensional homogeneous $G$-invariant subsets that fail to be $\epsilon$-dense in $(\Fb/\cO_F)^2$ (resp. $\Kb/\cO_K$).
\end{enumerate}
\end{lemma}

\begin{proof} In the $(\Fb/\cO_F)^2$ setting, all the properties above can be found in \cite{LW10}*{\S3 \& \S4}.

The corresponding statements for $\Kb/\cO_K$ immediately follow, thanks to the facts that the $G$-action commutes with Diagram \ref{commdiag} and that $\pi_\triangle$ is a finite-to-one continuous group morphism between two tori.
\end{proof}

The major new ingredient in our analysis is the classification in \cite{LW10} of all infinite proper $G$-invariant closed subsets in $(\Fb/\cO_F)^2$. When the unit rank is strictly greater than $2$, all of those are $s$-dimensional homogeneous $G$-invariant closed subsets:

\begin{proposition}\label{joining} Let $K$, $F$ and $G$ be as above. If the unit rank $r$ is at least $3$, then the following are true in both $(\Fb/\cO_F)^2$ and $\Kb/\cO_K$: \begin{enumerate}[(1)]
\item Every $G$-orbit closure is either a finite orbit consisting of rational points, or the whole space $(\Fb/\cO_F)^2$ (resp. $\Kb/\cO_K$), or a $s$-dimensional homogeneous $G$-invariant subset;
\item For all $\epsilon>0$, there is a finite union of $s$-dimensional homogeneous $G$-invariant subsets that contains all the rational points in $(\Fb/\cO_F)^2$ (resp. in $\Kb/\cO_K$) whose orbit fail to be $\epsilon$-dense.
\end{enumerate}
\end{proposition}

\begin{proof} Since $F$ is not CM, the proposition was proved for $(\Fb/\cO_F)^2$ in \cite{LW10}*{Theorem 3.15 \& 3.16}. Again this directly imply the same claims in $\Kb/\cO_K$.
\end{proof}

\subsection{Localized spectrum on invariant subsets}

We reduce the description of $\Spec(\Kb)$ (and that of $\Spec(K)$) to the study of the behavior of $N_*$ on certain affine subspaces of $\Fb^2$.

\begin{lemma}\label{affsub} Suppose $L\subset \Kb/\cO_K$ is an $s$-dimensional homogeneous $G$-invariant subset of slope $\varphi$. Then there is a finite subset $\Omega_L\subset F^2$ such that
\begin{equation}V^\varphi+\omega\subset\pi_{\cO_F^2}^{-1}\big(\pi_\triangle^{-1}(L)\big), \forall \omega\in\Omega_L\end{equation}
and for all $z\in L$, \begin{equation}\label{affsubstatement}m_{\Kb/\cO_K}(z)=\min_{\omega\in\Omega_L}\inf\big\{N_*(y):y\in\pi_{\cO_F^2}^{-1}\big(\pi_\triangle^{-1}(G.z)\big)\cap(V^\varphi+\omega)\big\}.\end{equation}
\end{lemma}

\begin{proof}
By Lemma \ref{homofacts}.(1), $L$ is a disjoint union $\bigsqcup_{a\in A}\big(\pi_\triangle(T^\varphi)+a\big)$ where $A\subset \Kb/\cO_K$ is a
finite collection of rational points and $\varphi\in F\cup\{\infty\}$. Hence the lift $\tilde L=\pi_{\cO_F^2}^{-1}\big(\pi_\triangle^{-1}(L)\big)$ is the union of all sets of the form $V^\varphi+\omega$ where $\omega$ takes value in $\pi_{\cO_F^2}^{-1}\big(\pi_\triangle^{-1}(A))$. The $(V^\varphi+\omega)$'s are parallel $s$-dimensional affine subspaces in the vector space $\Fb^2$. For two different $\omega$'s, the two $(V^\varphi+\omega)$'s either coincide or are disjoint. Fix a subset $\Omega$ from $\pi_{\cO_F^2}^{-1}\big(\pi_\triangle^{-1}(A))$ such that $\tilde L=\bigsqcup_{\omega\in\Omega}(V^\varphi+\omega)$ is a disjoint union. Remark the union is locally finite in the sense that any compact set in $\Fb^2$ intersects only finitely many such affine subspaces.

$\tilde L$ is $G$-invariant as the $G$-action commutes with $\pi_\triangle$ and $\pi_{\cO_F^2}$.

Suppose $z\in L$ and let $x$ be an arbitrary point from $\pi_{\cO_K}^{-1}(z)$. By Corollary \ref{finiteres}, we see that \begin{equation}\label{finitereslift}\begin{split}m_{\Kb/\cO_K}(z)=&m_\Kb(x)=\inf_{x'\in (G.x+\cO_K)\cap B}|N_\Kb(x')|
\\=&\inf_{x'\in \pi_{\cO_K}^{-1}(G.z)\cap B}|N_\Kb(x')|,\end{split}\end{equation} where $B\subset\Kb$ is a compact subset that depends only on $K$ and $G$.

Notice $\pi_{\cO_K}^{-1}(G.z)\subset \pi_{\cO_K}^{-1}(L)=\rho^{-1}(\tilde L)$. Therefore, \begin{equation}\label{compactF2}\begin{split}\pi_{\cO_K}^{-1}(G.z)\cap B=&\pi_{\cO_K}^{-1}(G.z)\cap \big(\rho^{-1}(\tilde L)\cap B\big)\\
=&\rho^{-1}\Big(\pi_{\cO_F^2}^{-1}\big(\pi_\triangle^{-1}(G.z)\big)\cap\big(\tilde L\cap \rho(B)\big)\Big).\end{split}\end{equation}

It follows from local finiteness that there is a finite disjoint decomposition \begin{equation}\label{affsub0}\tilde L\cap \rho(B)=\bigsqcup_{\omega\in\Omega_L}\big((V^\varphi+\omega)\cap \rho(B)\big)\end{equation} where $\Omega_L$ is a finite subset of $\Omega$. For each $\omega\in\Omega_L$ the component $(V^\varphi+\omega)\cap\rho(B)$ is a compact region of the affine subspace $V^\varphi+\omega$.

From (\ref{finitereslift}), (\ref{compactF2}) and (\ref{affsub0}), one can deduce that:
\begin{equation}\label{affsub1}\begin{split}&m_{\Kb/\cO_K}(z)\\
=&\inf\Big\{N_*(y):y\in\pi_{\cO_F^2}^{-1}
\big(\pi_\triangle^{-1}(G.z)\big)\cap\big(\bigsqcup_{\omega\in\Omega_L}(V^\varphi+\omega)\big)\cap\rho(B)\Big\}\\
=&\min_{\omega\in\Omega_L}\inf\big\{N_*(y):y\in\pi_{\cO_F^2}^{-1}\big(\pi_\triangle^{-1}(G.z)\big)\cap(V^\varphi+\omega)\cap\rho(B)\big\}.\end{split}\end{equation}
Here we used the fact that $N_K=N_*\circ\rho$, as well as that $N_*$ is, by definition (\ref{rhonorm}), non-negative.

This obviously implies that the right-hand side is bounded by the left-hand side in (\ref{affsubstatement}). In the other direction,
\begin{equation}\label{affsub2}\begin{split}
&\min_{\omega\in\Omega_L}\inf\big\{N_*(y):y\in\pi_{\cO_F^2}^{-1}\big(\pi_\triangle^{-1}(G.z)\big)\cap(V^\varphi+\omega)\big\}\\
\geq&\inf\big\{N_*(y):y\in\pi_{\cO_F^2}^{-1}\big(\pi_\triangle^{-1}(G.z)\big)\big\}\\
=&\inf\big\{\big|N_K\big(\rho^{-1}(y)\big)\big|:y\in\pi_{\cO_F^2}^{-1}\big(\pi_\triangle^{-1}(G.z)\big)\big\}\\
=&\inf\big\{|N_\Kb(x)|:x\in G.\pi_{\cO_K}^{-1}(z)\big\}=\inf\big\{|N_\Kb(x)|:x\in \pi_{\cO_K}^{-1}(z)\big\}\\
=&m_{\Kb/\cO_K}(z).\end{split}\end{equation}
Again, we used the commutativity between Diagram (\ref{commdiag}) and the $G$-action, as well as the $G$-invariance of $|N_\Kb(\cdot)|$. The proof is completed.
\end{proof}

On each $V^\varphi+\omega$, we can analyze explicitly the functional $N_*$.
\begin{lemma}\label{normhp}For any rational point $\omega\in F^2$ and all $\varphi\in F\cup\{\infty\}$, the restriction of $N_*$ to the affine subspace $V^\varphi+\omega$ has a minimum which is achieved by a rational point. Unless $0\in V^\varphi+\omega$, the minimum value is positive and the minimum point is unique .\end{lemma}

\begin{proof} Suppose first $\varphi\in F$ and $\omega=(\omega^{(1)},\omega^{(2)})\in F^2$. Then one can replace $\omega$ by $(\omega^{(1)}-\varphi\omega^{(2)},0)$ without changing $V^\varphi+\omega$. So we may assume without loss of generality  $\omega=(\beta,0)$ where $\beta\in F$.

In this case $V^\varphi+\omega$ can be identified with
\begin{equation}\label{affsubcoord}\{(y^{(1)},y^{(2)})\in\Fb^2: y^{(1)}=\varphi y^{(2)}+\beta\}.\end{equation}
In particular, for $y=(y^{(1)},y^{(2)})\in V^\varphi+w$, $y$ is uniquely determined by $y^{(2)}$ and
\begin{equation}\label{normfactor}\begin{split}N_*(y)=&\prod_{i=1}^s\left(\big((\varphi_i+\re{\eta_i})y^{(2)}_i+\beta_i\big)^2+(\im{\eta_i})^2(y^{(2)}_i)^2\right)\\
=&\prod_{i=1}^sf_i(y_i^{(2)}),\end{split}\end{equation}
with
\begin{equation}f_i(\theta)=(\varphi_i^2+2\varphi_i\re{\eta_i}+|\eta_i|^2)\theta^2+2(\varphi_i+\re{\eta_i})\beta_i\theta+\beta_i^2\end{equation}

It is clear that $f_i(\theta)\geq 0$ and has a minimum achieved at the unique point $-\dfrac{\big(\varphi_i+\re{\eta_i}\big)\beta}{\varphi_i^2+2\varphi_i\re{\eta_i}+|\eta_i|^2}$.

By Remark \ref{CMconj}, \begin{equation}\begin{split}
\re\eta_i=&\frac{\sigma_i(\eta)+\overline{\sigma_i(\eta)}}2=\frac{\sigma_i(\eta)+\sigma_i(\bar\eta)}2=\frac{\sigma_i(\eta+\bar\eta)}2\\
=&\frac{\sigma_i\big(\Tr_{K/F}(\eta)\big)}2=\frac{\tau_i\big(\Tr_{K/F}(\eta)\big)}2\\
=&\left(\frac{\Tr_{K/F}(\eta)}2\right)_i.\end{split}\end{equation}
And for similar reasons, $|\eta_i|^2=\big(N_{K/F}(\eta)\big)_i$.

So for the element $\xi=-\dfrac{\big(\varphi+\frac12\Tr_{K/F}(\eta)\big)}{\varphi^2+\varphi\Tr_{K/F}(\eta)+N_{K/F}(\eta)}\in F$ and each $i\in I$, $\xi_i$ is the unique point at which $f_i$ achieves its minimum. The minimum can be easily verified to be
\begin{equation}\begin{split}f_i(\xi_i)=&\left(\frac{\big(N_{K/F}(\eta)-\frac14\Tr_{K/F}^2(\eta)\big)\beta^2}{\varphi^2+\varphi\Tr_{K/F}(\eta)+N_{K/F}(\eta)}\right)_i\\
=&\frac{(\im{\eta_i})^2\beta_i^2}{\varphi_i^2+2\varphi\re{\eta_i}+|\eta_i|^2}.\end{split}\end{equation}

As $\beta\in F$ is rational, $\beta_i=0$ if and only if $\beta=0$. Moreover for any $i\in I$, because $\im{\eta_i}\neq 0$, $f_i(\xi_i)=0$ if and only if $\beta=0$.

It follows that the restriction of $N_*$ to $V^\varphi+\omega$ has a minimum point at the point $y=(\varphi\xi+\beta,\xi)$.  Moreover, if $\beta\neq0$, then the minimum values $f_i(\xi_i)$ are all positive; and thus $N_*(y)>0$ and the minimum point $y$ is unique. Otherwise, $\beta=0$ and $V^\varphi+\omega$ contains $0$ by (\ref{affsubcoord}). This finishes the proof for $\varphi\in F$.

It remains to check what happens when the slope $\varphi$ is $\infty$, in which case $V^\infty+\omega=\{(y^{(1)},y^{(2)})\in\Fb^2: y^{(2)}=\omega^{(2)}\}$ and $y$ is uniquely determined by $y^{(1)}$. Denote $\beta=\omega^{(2)}\in F$. In this case, $N_*(y)$ can be decomposed as $\prod_{i=1}^sf_i(y^{(1)})$, where $f_i$ is a new polynomial given by $f_i(\theta)=\theta^2+2(\re{\eta_i})\beta_i\theta+|\eta_i|^2\beta_i^2$. Similar analysis as in the $\varphi\in F$ case shows that $f_i$ has a unique minimum point $\xi_i$, which is the embedding of $\xi=-\Tr_{K/F}(\eta)\cdot\beta\in F$ into $\bR$ by $\tau_i$. And the minimum value is $f_i(\xi_i)=(\im\eta_i)^2\beta_i^2$. Again, if $\beta\neq 0$, then $f_i(\xi_i)>0$ for all $i$ and $y$ is the unique minimum point for the product form $N_*(y)$; otherwise $0\in V^\infty+\omega$.
\end{proof}

Next, we study the localized Euclidean spectrum on each individual $s$-dimensional homogeneous $G$-invariant subset.

For any subset $A$ of $\Kb/\cO_K$, write
\begin{equation}\label{KbGammaSpec}\Spec_{\Kb/\cO_K}(A)=\{m_{\Kb/\cO_K}(z):z\in A\}.\end{equation}

\begin{proposition}\label{subspec} Suppose $r\geq 2$ and $L\subset\Kb/\cO_K$ is an $s$-dimensional homogeneous $G$-invariant subset, then $\Spec_{\Kb/\cO_K}(L)$ is a subset of $\bQ$ and can be written as $\{\nu,\mu_1,\mu_2,\cdots\}$ where: \begin{enumerate}[(1)]
\item $\nu=0$ if and only if $0\in L$;
\item $\{y\in L:m_{\Kb/\cO_K}(y)=\nu\}$ is the union of the set of all irrational points in $L$ and a non-empty finite set of rational points;
\item $\{y\in L:m_{\Kb/\cO_K}(y)=\mu_n\}$ is a finite non-empty set of rational points for all $n\geq 1$;
\item $\mu_1>\mu_2>\cdots$ and $\lim_{n\rightarrow\infty}\mu_n=\nu$;
\end{enumerate}
\end{proposition}

In particular, the proposition implies that every value from $\Spec_{\Kb/\cO_K}(L)$ is achieved by at least one, but finitely many, rational point from $L$.

\begin{proof} {\noindent\bf (1)} Denote by $\varphi$ the slope of $L$. Lemma \ref{affsub} implies that $\inf_{z\in L}m_{\Kb/\cO_K}(z)$ equals $\min_{\omega\in\Omega_L}\inf\{N_*(y):y\in V^\varphi+\omega\}$ where $\Omega_L\subset F^2$ is a finite set of rational point decided by $L$. By Lemma \ref{normhp}, for each $\omega\in\Omega_L$, there is $y_\omega\in V^\varphi+\omega$ such that $N_*(y_\omega)=\min\{N_*(y):y\in V^\varphi+\omega\}$. Hence $\inf_{z\in L}m_{\Kb/\cO_K}(z)=\min_{\omega\in\Omega_L}N_*(y_\omega)$. Denote this minimum by $\nu$.

There is at least one of the $y_\omega$'s, which we denote by $y_\nu$, such that $N_*(y_\omega)=\nu$. Then $z_\nu=\pi_\triangle\big(\pi_{\cO_F^2}(y_\nu)\big)\in L$ by Lemma \ref{affsub}, and thus $m_{\Kb/\cO_K}(z_\nu)\leq\nu$ by definition and in consequence $m_{\Kb/\cO_K}(z_\nu)=\nu$. Note $z_\nu$ is rational, hence $\nu= m_{\Kb/\cO_K}(z_\nu)\in\bQ$ by Lemma \ref{toralmin}.(4).

By Lemma \ref{normhp}, $\nu=0$ if and only if $y_\nu=0$, or equivalently $z_\nu=0$. On the other hand if $L$ passes through $0$ then clearly $m_{\Kb/\cO_K}(0)=0$ is the minimum of $m_{\Kb/\cO_K}$ on $L$. This proves claim (1).

{\noindent\bf (2)} Assume $z\in L$ is irrational, then by Lemma \ref{homofacts}.(2), $G.z$ is dense in $L$. It follows from Lemma \ref{toralmin}.(3) that $m_{\Kb/\cO_K}(z)=\nu$.

When $\nu=0$, the only rational point $z\in L$ with $m_{\Kb/\cO_K}(z)=0$ is $0$. Assuming $\nu>0$, we try to show that all rational points $z\in L$ such that $m_{\Kb/\cO_K}(z)=\nu$ are contained in a fixed finite set.

Let $z$ be such a point. By Lemma \ref{affsub}, there exists $\omega\in\Omega_L$ such that \begin{equation}\label{subspecinf}\inf\big\{N_*(y):y\in\pi_{\cO_F^2}^{-1}\big(\pi_\triangle^{-1}(G.z)\big)\cap(V^\varphi+\omega)\big\}=\nu.\end{equation} Note for all $y\in\pi_{\cO_F^2}^{-1}\big(\pi_\triangle^{-1}(G.z)\big)\cap(V^\varphi+\omega)$,  $\pi_{\cO_K}\big(\rho^{-1}(y)\big)=z$ and hence $N_*(y)=N_\Kb\big(\rho^{-1}(y)\big)$ takes values from a discrete set of rational numbers as we have seen in the proof of Lemma \ref{toralmin}.(4). Therefore, the infimum is actually a minimum. In other words, there is $y\in \pi_{\cO_F^2}^{-1}\big(\pi_\triangle^{-1}(G.z)\big)$ such that the infimum in (\ref{subspecinf}), which equals $\nu$, is attained at $y$. Since $\nu>0$, by Lemma \ref{normhp}, $V^\varphi+\omega$ doesn't contain $0$ and $y$ must be $y_\omega$. Thus $z\in G.\pi_\triangle\big(\pi_{\cO_F^2}(y_\omega)\big)$. Because $y_\omega$ is rational, this is a finite $G$-orbit. So the finite set $\bigcup_{\omega\in\Omega_L}G.\pi_\triangle\big(\pi_{\cO_F^2}(y_\omega)\big)$ covers all rational points $z\in L$ at which $m_{\Kb/\cO_K}$ equals $\nu$. This establishes Part (2).

{\noindent\bf (3)} There are infinitely many rational points in $L$. So it follows from the finiteness proved above that $\Spec_{\Kb/\cO_K}$ contains values other than $\nu$.

Therefore by upper semicontinuity, $L_{\geq \nu+\delta}=\{z\in L: m_{\Kb/\cO_K}(z)\geq\nu+\delta\}$ is a proper non-empty closed subset of $L$ for all sufficiently small positive $\delta$. Moreover, it is $G$-invariant by Lemma \ref{toralmin}.(1). By the remark at the beginning of part (2) above, $L_{\geq\nu+\delta}$ consists of rational points. Moreover, $L_{\geq\nu+\delta}$ is finite. Actually, suppose $L_{\geq\nu+\delta}$ is infinite then it contains an $s$-dimensional homogeneous $G$-invariant subset $L'$ by Lemma \ref{homofacts}.(3). However $L'$ must contain irrational points, which contradicts the rationality of points from $L_{\geq\nu+\delta}$. Hence we conclude that $L_{\geq \nu+\delta}$ is non-empty finite subset of rational points in $L$ for tiny $\delta$.

For any $\mu\in \Spec_{\Kb/\cO_K}\backslash\{\nu\}$, we know $\mu>\nu$ and denote by $L_{=\mu}$ the set $\{z\in L: m_{\Kb/\cO_K}(z)=\mu\}$. Then $L_{=\mu}$ is a subset of $L_{\geq\nu+\delta}$ for $\delta\in(0,\mu-\nu)$ and in consequence consists of a finite number of rational points.

{\noindent\bf (4)} Observe that the collection of rational points in $L$, which is infinite, is the union of $\{z\in L: z\text{ is rational, }m_{\Kb/\cO_K}(z)=\nu\}$ and all the $L_{=\mu}$'s where $\mu\in \Spec_{\Kb/\cO_K}\backslash\{\nu\}$. We have already seen that each of these sets is finite, therefore $\Spec_{\Kb/\cO_K}\backslash\{\nu\}$ must be infinite.

Fix $z_\mu\in L_{=\mu}$, then $\mu=m_{\Kb/\cO_K}(z_\mu)\in\bQ$ by Lemma \ref{toralmin}.(4). As it was already verified that $\nu\in\bQ$, this asserts that \begin{equation}\Spec_{\Kb/\cO_K}(L)\subset\bQ.\end{equation}

Furthermore, the spectrum has no accumulation point greater than $\nu$. Otherwise, for sufficiently small $\delta$, there are infinitely many values in $\Spec_{\Kb/\cO_K}(L)\cap[\nu+\delta,\infty)$. Since each of these values correspond to at least one point in $L_{\geq\nu+\delta}$, it follows that $L_{\geq\nu+\delta}$ is infinite, which contradicts the previous conclusion.

In addition, recall that $m_{\Kb/\cO_K}(z)\leq M(\Kb)\leq 2^{-d}D_K$ by \ref{EMupper}. So $\Spec_{\Kb/\cO_K}(L)\backslash\{\nu\}$ is a bounded infinite subset of $\bQ\cap(\nu,\infty)$ and has no accumulation point other than $\nu$. The only possibility is a decreasing sequence approaching $\nu$, which is Part (4) of the lemma.\end{proof}

\subsection{Proof of main results}

We are now able to establish a complete characterization of the Euclidean and inhomogeneous spectra of $K$ in case that $r\geq 3$ by putting pieces together.

\begin{theorem}\label{describespec} Suppose $K$ is a CM number field of unit rank $3$ or higher, then the inhomogeneous and Euclidean spectra coincide: $\Spec(\Kb)=\Spec(K)$. Moreover, $\Spec(K)$ is a countable subset of $\bQ$ and can be decomposed as $\{0\}\sqcup\left(\bigsqcup_{n=1}^\infty S_n\right)$, where:

\begin{enumerate}[(1)]
\item For all $n\geq 1$, $S_n$ can be written as $\{\nu_n,\mu_{n,1},\mu_{n,2},\cdots\}$ such that:\begin{itemize}
\item $\mu_{n,1}>\mu_{n,2}>\cdots$,
\item $\lim_{k\rightarrow\infty}\mu_{n,k}=\nu_n$,
\item $\lim_{n\rightarrow\infty}\nu_n=0$,
\item $\nu_n>\mu_{n+1,1}$;
\end{itemize}

\item For each $\mu_{n,k}$, $\{z\in\Kb/\cO_K: m_{\Kb/\cO_K}(z)=\mu_{n,k}\}$ is a finite subset of rational points.

For each $\nu_n$, $\{z\in K/\cO_K: m_{\Kb/\cO_K}(z)=\nu_n\}$ is finite and non-empty. And the set  $\{z\in\Kb/\cO_K: z\text{ is irrational, }m_{\Kb/\cO_K}(z)=\nu_n\}$ is the set of all irrational points in a certain finite union of $s$-dimensional affine subtori where $s=\frac12 \deg K$.

\end{enumerate}
\end{theorem}

\begin{proof} {\noindent\bf Step 1.} Construct $\eta$, $\rho$, $F$, $\Gamma$ as in previous discussions. Then $s=\deg F\geq 4$ and we will be able to make use of Proposition \ref{joining}.

Let $\cE$ be the collection of all $s$-dimensional homogeneous $G$-invariant subsets $L\subset\Kb/\cO_K$ that avoid $0$. For each $L\in\cE$, denote $\nu_L=\min\Spec_{\Kb/\cO_K}(L)$, which exists by Proposition \ref{subspec}.

We classify all points $z\in\Kb/\cO_K$ into several categories:\begin{itemize}
\item[(Ia)] $z=0$;
\item[(Ib)] $z$ is irrational and is not contained in any $L\in\cE$;
\item[(IIa)] $z$ is irrational and belongs to some $L\in\cE$;
\item[(IIb)] $z$ is rational and there is $L\in\cE$, which may or may not contain $z$, such that $m_{\Kb/\cO_K}(z)=\nu_L$;
\item[(III)] $z$ is a non-zero rational point that doesn't fall into category (IIb).
\end{itemize}

These types obviously exhaust all points in $\Kb/\cO_K$.

First, if $z$ is of type (Ia) or (Ib), then $m_{\Kb/\cO_K}(z)=0$. This is obviously true if $z=0$. If $z$ is of type (Ib), then by Proposition \ref{joining}, $\overline{G.z}$ is either $\Kb/\cO_K$ or an $s$-dimensional homogeneous $G$-invariant subset that contains $0$. In both cases $0\in\overline{G.z}$ and by Lemma \ref{toralmin}, $m_{\Kb/\cO_K}(z)$ vanishes.

Second, if $z$ belongs to category (IIa) or (IIb) and is associated with $L\in\cE$, then $m_{\Kb/\cO_K}(z)=\nu_L$. For points of type (IIb) this is part of construction. If $z$ is irrational and $z\in L$, then this is a consequence of Lemma \ref{subspec}.(2) instead.

{\noindent\bf Step 2.} We show that for all $\delta>0$, there are only finitely many $L\in\cE$ such that $\nu_L>\delta$.

Actually, by Lemma \ref{toralmin}, for each of these $L$'s and $z\in L$, $\|z\|>\delta^\frac1d$. In other words, $L$ fails to be $\delta^\frac1d$-dense in $\Kb/\cO_K$, the claim follows from Lemma \ref{homofacts}.(4).

{\noindent\bf Step 3.} We claim $\cE$ is infinite. Actually, there are infinitely many $L\in\cE$ that have slope $\infty$.

Since in $\Kb/\cO_K$ the only $s$-dimensional homogeneous $G$-invariant subset with slope $\infty$ that contains $0$ is $\pi_\triangle(T^\infty)$. It is enough to show there are infinitely many $s$-dimensional homogeneous $G$-invariant subsets that have slope $\infty$.

Each rational point $\gow\in\Fb/\cO_F$ gives rise to a $s$-dimensional homogeneous $G$-invariant subset $\{(\goz^{(1)},\goz^{(2)}):\goz^{(2)}=\gow\}$ in $(\Fb/\cO_F)^2$, and the correspondence is one-to-one. Since $\Fb/\cO_F$ contains infinitely many rational points, there are infinitely many $s$-dimensional homogeneous $G$-invariant subsets of slope $\infty$ in $(\Fb/\cO_F)^2$, each of these projects to an $s$-dimensional homogeneous $G$-invariant subset of slope $\infty$ in $\Kb/\cO_F$ under $\pi_\triangle$. This establishes the claim, as $\pi_\triangle$ is a finite covering map.

{\noindent\bf Step 4.} The set $\cA=\{\nu_L:L\in\cE\}$ can be reordered into a strictly decreasing sequence of rational numbers $\nu_1>\nu_2>\cdots$ that converges to $0$.

To prove this it suffices to show $\cA$ is a bounded infinite set of positive rational numbers and has no accumulation point other than $0$.

The boundedness follows from that of $m_{\Kb/\cO_K}$. The positivity and rationality are confirmed by Proposition \ref{subspec}. Hence it suffices to show $\cA$ is infinite, and has no accumulation point other than $0$.

If $\cA$ is finite, by infinity of $\cE$ there must be an infinite family of $L\in\cE$ such that the corresponding $\nu_L$'s are the same number, say $\nu$.

On the other hand, if $\cA$ has a non-zero accumulation point, which must be positive as each $\nu_L$ is, then there is a sequence of $L_n$'s such that $\nu_{L_n}$'s are distinct and converge to a positive value $\nu$. Without loss of generality, one may assume $\nu_{L_n}>\frac\nu2>0$.

Thus in both cases, there are an infinity of different $L$'s from $\cE$ such that $\nu_L>\frac\nu2>0$, which is impossible by Step 2. This completes the proof of the claim.

{\noindent\bf Step 5.} Write $\cS_1=\Spec(\Kb)\cap(\nu_1,\infty)$ and $\cS_n=\Spec(\Kb)\cap(\nu_n,\nu_{n-1})$ for $n\geq 2$. One wants to show that each $\cS_n$ can be written as a decreasing sequence of rational numbers $\mu_{n,1}>\mu_{n,2}>\cdots$ that converges to $\nu_n$, and that $\{z\in\Kb/\cO_K:m_{\Kb/\cO_K}(z)=\mu\}$ is a finite set of rational points for all $\mu\in\cS_n$.

In order to show the first half of the claim, it suffices to show $\cS_n$ is a bounded infinite set of rational numbers and has no accumulation point other than $\nu_n$.

Boundedness is again easily guaranteed. By (\ref{KbSpec}) and Step 1, any $\mu\in\cS_n$ can be achieved by $m_{\Kb/\cO_K}$ only at points of type (III), which are rational. Hence by Lemma \ref{toralmin}.(4), $\cS_n\subset\bQ$.

Since $\nu_n\in\cA$, there is $L$ such that $\nu_n=\nu_L$. Lemma \ref{subspec} asserts that there is a decreasing sequence from $\Spec_{\Kb/\cO_K}(L)\subset\Spec(\Kb)$ whose limit is $\nu_n$. In particular, this implies the infiniteness of $\cS_n$.

So what remains to be done is to get a contradiction assuming that: either $\cS_n$ has an accumulation point $\nu'$ which doesn't equal $\nu_n$, or the set $\{z\in\Kb/\cO_K:m_{\Kb/\cO_K}(z)=\mu\}$, which we just showed consists of rational points, is infinite for some $\mu\in\cS_n$.

In the first case, $\nu'>\nu_n$ and there are a sequence of rational points $z_k\in K/\cO_K$ such that the $m_{\Kb/\cO_K}(z_k)$'s are all different and converge to $\nu'$.

In the second case, let $\nu'=\mu$. Then in both cases we have an infinite sequence of distinct rational points $z_k$ such that $\lim_{k\rightarrow\infty}m_{\Kb/\cO_K}(z_k)=\nu'$. In particular, we may assume $m_{\Kb/\cO_K}(z_k)>\frac{\nu'}2$ for all $k$.

By Lemma \ref{toralmin}.(4), the orbits $G.z_k$ don't meet the neighborhood of radius $(\frac{\nu'}2)^{\frac1d}$ of the origin, and hence Proposition \ref{joining}.(2) implies that the $z_k$'s are contained in a finite union of $s$-dimensional homogeneous $G$-invariant subsets. Without loss of generality, assume they are all in the same $s$-dimensional homogeneous $G$-invariant subset $L'$. By Proposition \ref{subspec}, when $z_k\in L'$ are all different, the only possible limit of $m_{\Kb/\cO_K}(z_k)$ as $k$ tends to $\infty$ is $\nu_{L'}$ and the $m_{\Kb/\cO_K}(z_k)$'s are all greater than or equal to $\nu_{L'}$. Thus $\nu_{L'}=\nu'>\nu_n>0$. On the other hand, $\nu'<\nu_{n-1}$ if $n\geq 2$. Moreover, by Proposition \ref{subspec}.(1), $0\notin L'$; in other words, $L'\in\cE$. Hence $\nu'=\nu_{L'}\in\cA$. But $\cA$ contains no value in $(\nu_n,\nu_{n-1})$ when $n\geq 2$, or in $(\nu_1,\infty)$. Therefore we obtain a contradiction and this completes Step 5.

{\noindent\bf Final Step.} Part (1) of the theorem results from (\ref{KbSpec}) and Steps 4 and 5. A corollary to it is that $\Spec(\Kb)\subset\bQ$ and is countable.

The first half of Part (2) was already proved in Step 5. We now prove the second half that involves the $\nu_n$'s.

By definition of $\cA$, each $\nu$ is equal to $\nu_L$ for at least one $L\in\cE$ and is therefore achieved by $m_{\Kb/\cO_K}$ at some rational point in $L$. Hence $Y_n=\{z\in K/\cO_K: m_{\Kb/\cO_K}(z)=\nu_n\}$ is non-empty. By Lemma \ref{toralmin}.(3), for all $z\in Y_n$, the $G$-orbit of $z$ avoids the neighborhood of radius $\nu_n^{\frac1d}$ around $0\in\Kb/\cO_K$. Hence we know from Proposition \ref{joining}.(2) that $Y_n$ is covered by a finite union of $s$-dimensional homogeneous $G$-invariant subsets. Moreover, by Proposition \ref{subspec}, $m_{\Kb/\cO_K}$ can take value $\nu_n$ at only finitely many rational points from any given $s$-dimensional homogeneous $G$-invariant subset. Thus $Y_n$ is finite.

Regarding the set $Y'_n=\{z\in\Kb/\cO_K: z\text{ irrational, }m_{\Kb/\cO_K}(z)=\nu_n\}$. By Step 1, it consists only of irrational points of type (IIa) and is the set of irrational points from all the $L\in\cE$ such that $\nu_L=\nu_n$. Because each $L$ is a finite union of $s$-dimensional affine subtori, it suffices to notice there are only finitely many such $L$'s, which follows from Step 2 by taking $\delta\in(0,\nu_n)$. This completes the proof of Part (2).

Last, Part (2) of the theorem confirms that any value from the inhomogeneous spectrum $\Spec(\Kb)$ can be achieved by $m_{\Kb/\cO_K}$ at some $z\in K/\cO_K$, and thus $\Spec(K)=\Spec(\Kb)$ by (\ref{KbSpec}) and(\ref{KSpec}).
\end{proof}

We are now at a position to prove Corollaries \ref{isolatta} and \ref{computable}.

\begin{proof}[Proof of Corollary \ref{isolatta}] The Euclidean minimum $M(K)$ is exactly the value $\mu_{1,1}$ from Theorem \ref{describespec}, therefore the desired isolatedness and finiteness follow directly from the theorem. On the other hand, to show $M(\Kb)$ is attained, it suffices to prove that for all $z\in\Kb/\cO_K$ such that $m_{\Kb/\cO_K}(z)=M(\Kb)$, there is a lift $x\in\pi_{\cO_K}^{-1}(z)$ such that $|N_K(x)|=M(\Kb)$. Because such a point $z$ must be rational by Theorem \ref{describespec},  Lemma \ref{toralmin}.(4) implies that $M(K)$ is attained in $\Spec(\Kb)$.
\end{proof}

\begin{proof}[Proof of Corollary \ref{computable}] For a detailed explanation of Cerri's algorithm, see \cite{CThesis}*{Chapter 3}. In Proposition 4.25 of that thesis, Cerri showed the algorithm stops in finite time for non-CM fields of unit rank at least $2$. However, the only facts he used were that $M(K)$ is isolated in $\Spec(\Kb)$ and that $\{z\in\Kb/\cO_K: m_{\Kb/\cO_K}(z)=M(K)\}$ is a finite set of rational points. Therefore thanks to Corollary \ref{isolatta}, the same proof is valid for CM fields of unit rank $3$ or higher.\end{proof}

{\noindent\bf Acknowledgments.} We are grateful to Elon Lindenstrauss for helpful and motivating discussions.


\begin{bibdiv}
\begin{biblist}
\bib{B06}{article}{
   author={Bayer Fluckiger, Eva},
   title={Upper bounds for Euclidean minima of algebraic number fields},
   journal={J. Number Theory},
   volume={121},
   date={2006},
   number={2},
   pages={305--323},
   issn={0022-314X},
   review={\MR{2274907 (2008a:11139)}},
}

\bib{BS52}{article}{
   author={Barnes, E. S.},
   author={Swinnerton-Dyer, H. P. F.},
   title={The inhomogeneous minima of binary quadratic forms. II},
   journal={Acta Math.},
   volume={88},
   date={1952},
   pages={279--316},
   issn={0001-5962},
   review={\MR{0054654 (14,956a)}},
}
	

\bib{B04}{article}{
   author={Belabas, Karim},
   title={Topics in computational algebraic number theory},
   journal={J. Th\'eor. Nombres Bordeaux},
   volume={16},
   date={2004},
   number={1},
   pages={19--63},
   issn={1246-7405},
   review={\MR{2145572 (2006a:11174)}},
}

\bib{B83}{article}{
   author={Berend, Daniel},
   title={Multi-invariant sets on tori},
   journal={Trans. Amer. Math. Soc.},
   volume={280},
   date={1983},
   number={2},
   pages={509--532},
   issn={0002-9947},
   review={\MR{716835 (85b:11064)}},
}

\bib{BLMV09}{article}{
   author={Bourgain, Jean},
   author={Lindenstrauss, Elon},
   author={Michel, Philippe},
   author={Venkatesh, Akshay},
   title={Some effective results for $\times a\times b$},
   journal={Ergodic Theory Dynam. Systems},
   volume={29},
   date={2009},
   number={6},
   pages={1705--1722},
   issn={0143-3857},
   review={\MR{2563089 (2011e:37022)}},
}


\bib{CL98}{article}{
   author={Cavallar, Stefania},
   author={Lemmermeyer, Franz},
   title={The Euclidean algorithm in cubic number fields},
   conference={
      title={Number theory},
      address={Eger},
      date={1996},
   },
   book={
      publisher={de Gruyter},
      place={Berlin},
   },
   date={1998},
   pages={123--146},
   review={\MR{1628838 (99e:11137)}},
}

\bib{CThesis}{article}{
   author={Cerri, Jean-Paul},
   title={Spectres euclidiens et inhomogènes des corps de nombres},
   journal={PhD Thesis, Universit\'{e} Nancy I},
   volume={},
   date={2005},
   pages={},

}

\bib{C06}{article}{
   author={Cerri, Jean-Paul},
   title={Inhomogeneous and Euclidean spectra of number fields with unit
   rank strictly greater than 1},
   journal={J. Reine Angew. Math.},
   volume={592},
   date={2006},
   pages={49--62},
}

\bib{C07}{article}{
   author={Cerri, Jean-Paul},
   title={Euclidean minima of totally real number fields: algorithmic
   determination},
   journal={Math. Comp.},
   volume={76},
   date={2007},
   number={259},
   pages={1547--1575 (electronic)},
   issn={0025-5718},
   review={\MR{2299788 (2008d:11143)}},
}

\bib{EL04}{article}{
   author={Einsiedler, Manfred},
   author={Lind, Douglas},
   title={Algebraic $\bZ^d$-actions on entropy rank one},
   journal={Trans. Amer. Math. Soc.},
   volume={356},
   date={2004},
   number={5},
   pages={1799--1831 (electronic)},
   issn={0002-9947},
   review={\MR{2031042 (2005a:37009)}},
}

\bib{FJ10}{article}{
   author={Fontein, Felix},
   author={Jacobson, Michael},
   title={Rigorous Computation of Fundamental Units in Algebraic Number Fields},
   journal={},
   volume={},
   date={2010},
   pages={preprint},
   issn={},
   review={},
}

\bib{GL87}{book}{
   author={Gruber, P. M.},
   author={Lekkerkerker, C. G.},
   title={Geometry of numbers},
   series={North-Holland Mathematical Library},
   volume={37},
   edition={2},
   publisher={North-Holland Publishing Co.},
   place={Amsterdam},
   date={1987},
   pages={xvi+732},
   isbn={0-444-70152-4},
   review={\MR{893813 (88j:11034)}},
}

\bib{HW79}{book}{
   author={Hardy, G. H.},
   author={Wright, E. M.},
   title={An introduction to the theory of numbers},
   edition={5},
   publisher={The Clarendon Press Oxford University Press},
   place={New York},
   date={1979},
   pages={xvi+426},
   isbn={0-19-853170-2},
   isbn={0-19-853171-0},
   review={\MR{568909 (81i:10002)}},
}


\bib{L95}{article}{
   author={Lemmermeyer, Franz},
   title={The Euclidean algorithm in algebraic number fields},
   journal={Exposition. Math.},
   volume={13},
   date={1995},
   number={5},
   pages={385--416},
   issn={0723-0869},
   review={\MR{1362867 (96i:11115)}},
}

\bib{LW10}{article}{
   author={Lindenstrauss, Elon},
   author={Wang, Zhiren},
   title={Topological self-joinings of Cartan actions by toral automorphisms},
   journal={},
   volume={},
   date={2010},
   pages={preprint},
   issn={},
   review={},
}

\bib{P75}{article}{
   author={Parry, Charles J.},
   title={Units of algebraic number fields},
   journal={J. Number Theory},
   volume={7},
   date={1975},
   number={4},
   pages={385--388},
}

\bib{S91}{article}{
   author={Sands, Jonathan W.},
   title={Generalization of a theorem of Siegel},
   journal={Acta Arith.},
   volume={58},
   date={1991},
   number={1},
   pages={47--57},
   issn={0065-1036},
   review={\MR{1111089 (92c:11134)}},
}

\bib{S95}{book}{
   author={Schmidt, Klaus},
   title={Dynamical systems of algebraic origin},
   series={Progress in Mathematics},
   volume={128},
   publisher={Birkh\"auser Verlag},
   place={Basel},
   date={1995},
   pages={xviii+310},
   isbn={3-7643-5174-8},
   review={\MR{1345152 (97c:28041)}},
}

\bib{V85}{book}{
   author={van der Linden, F. J.},
   title={Euclidean rings with two infinite primes},
   series={CWI Tract},
   volume={15},
   publisher={Stichting Mathematisch Centrum Centrum voor Wiskunde en
   Informatica},
   place={Amsterdam},
   date={1985},
   pages={vii+200},
   isbn={90-6196-286-2},
   review={\MR{787609 (87a:11107)}},
}		

\bib{V96}{article}{
   author={Voutier, Paul},
   title={An effective lower bound for the height of algebraic numbers},
   journal={Acta Arith.},
   volume={74},
   date={1996},
   number={1},
   pages={81--95},
   issn={0065-1036},
   review={\MR{1367580 (96j:11098)}},
}

\bib{W10}{article}{
   author={Wang, Zhiren},
   title={Quantitatitve density under higher rank abelian algebraic toral actions},
   journal={Int. Math. Res. Not.},
   volume={16},
   date={2011},
   pages={3744-3821},
   issn={},
}

\bib{W97}{book}{
   author={Washington, Lawrence C.},
   title={Introduction to cyclotomic fields},
   series={Graduate Texts in Mathematics},
   volume={83},
   edition={2},
   publisher={Springer-Verlag},
   place={New York},
   date={1997},
   pages={xiv+487},
   isbn={0-387-94762-0},
   review={\MR{1421575 (97h:11130)}},
}




\end{biblist}
\end{bibdiv}
\end{document}